\documentclass[final,3p]{elsarticle}

\usepackage{graphicx}
\usepackage{amsmath}
\usepackage{amsfonts}
\usepackage{amssymb}
\usepackage{amsthm}
\usepackage{overpic}

\newtheorem{theorem}{Theorem}
\newtheorem{conjecture}[theorem]{Conjecture}
\newtheorem{proposition}[theorem]{Proposition}
\newtheorem{remark}[theorem]{Remark}
\newtheorem{problem}[theorem]{Problem}

\newtheorem{definition}[theorem]{Definition}
\newtheorem{example}[theorem]{Example}
\newtheorem*{definition*}{Definition}
\newtheorem*{example*}{Example}
\newtheorem{corollary}[theorem]{Corollary}

\usepackage[pdftex,%
            colorlinks,linkcolor=red,citecolor=blue,urlcolor=blue,%
            pdfstartview=FitH,%
            linktocpage,bookmarks=false
           ]{hyperref}
\usepackage{anysize}
\usepackage{url}

\marginsize{1.4 cm}{1.4 cm}{0.8 cm}{0.8 cm}
\usepackage{setspace}
\usepackage{longtable}

\usepackage{pgf,tikz}
\usetikzlibrary{arrows}

\newcommand{\mscomm}[1]{\begingroup\textbf{\color{green}#1}\endgroup}

\definecolor{drot}{rgb}{0.7,0,0.1}
\newcommand{\cre}[1]{{\color{drot} #1}}

\newcommand{\bd}{\mathbf d}
\newcommand{\bv}{\mathbf v}
\newcommand{\bn}{\mathbf n}
\newcommand{\bnt}{\tilde{\mathbf n}}

\begin{document}

\begin{frontmatter}

\title{Characterizing envelopes of moving rotational cones and applications in CNC machining}

	
\author[HSE,RAS]{Mikhail Skopenkov\corref{cor1}}	
\ead{mikhail.skopenkov@gmail.com}
\author[Harbin]{Pengbo Bo}
\ead{pbbo@hit.edu.cn}
\author[BCAM]{Michael Barto\v{n}}
\ead{mbarton@bcamath.org}
\author[KAUST]{Helmut Pottmann}	
\ead{helmut.pottmann@kaust.edu.sa}
\cortext[cor1]{Corresponding author}

\address[HSE]{National Research University Higher School of Economics, Faculty of Mathematics, Usacheva 6, Moscow 119048, Russia}
\address[RAS]
{Institute for Information Transmission Problems of the Russian Academy of Sciences, Bolshoy Karetny 19 bld.1, Moscow 127051, Russia}
\address[Harbin]{School of Computer Science and Technology,
Harbin Institute of Technology, West Wenhua Str. 2, 264209 Weihai, China}
\address[BCAM]{BCAM -- Basque Center for Applied Mathematics,  Alameda de Mazarredo 14, 48009 Bilbao, Basque Country, Spain}
\address[KAUST]{ King Abdullah University of Science and Technology, P.O. Box 2187, 4700 Thuwal, 23955-6900, Kingdom of Saudi Arabia}

\begin{abstract}

Motivated by applications in CNC machining, we provide a characterization of surfaces which are enveloped by a one-parametric family of congruent rotational cones.  As limit cases, we also address developable surfaces and ruled surfaces.
The characterizations are higher order nonlinear PDEs generalizing the ones by Gauss and Monge
for developable surfaces and ruled surfaces, respectively.
The derivation includes results on local approximations of a surface by cones of revolution, which are expressed
by contact order in the space of planes. These results are themselves of interest in geometric computing, for example in cutter selection and positioning for flank CNC machining.

\end{abstract}


\end{frontmatter}

\section{Introduction}\label{sec-intro}

Various manufacturing technologies, such as hot wire cutting, electrical discharge machining or computer numerically controlled (CNC) machining are based on a moving \emph{tool}, the active part of which can be a curve or a surface. They generate surfaces which are swept by a simple curve, e.g. a straight line segment or a circular arc, or are enveloped by a simple surface. The latter
case mostly refers to CNC machining where the moving tool is part of a rotational surface (sphere, rotational cylinder, rotational cone, torus). In order to produce a given shape with such a manufacturing process, one has to approximate the target shape by surfaces which are generated by a moving tool of the available type. Depending on the application,  such an approximation has to be highly accurate and, for example in the case of CNC machining may have to meet a numerical tolerance of a few micrometers for objects of the size of tens of centimeters. Such high precision pushes demands on the path-planning algorithms which greatly benefit from a higher order analysis of the contact between the reference surface and the surface generated by
the moving tool.

With the flank CNC machining application in mind, we present such an analysis for envelopes of rotational cones. In order to obtain contact of order $n$ between an envelope of a moving rotational
cone and a design surface $\Phi$, it is not necessary that each position of the cone has contact of order $n$ with
$\Phi$, when viewing the surfaces as point sets. This is obvious anyway, since 2nd order contact between  a cone and a surface $\Phi$ would already imply
vanishing Gaussian curvature of $\Phi$, i.e., a developable surface $\Phi$. One needs contact of order $n$ between the cone and the surface, viewed in the space of planes.  It is related to the fact that a
cone possesses just a
one-parameter family of tangent planes. This indicates the advantage of using a geometry, in which the (oriented) planes in Euclidean space
are the basic elements. Therefore, we use Laguerre geometry and work in a point model of the set of oriented planes, known as the \emph{isotropic model} of Laguerre geometry. There, a cone appears
as a curve (an isotropic circle) and not as a surface. That is, the analysis of cone-surface
contact is transferred to the study of a curve-surface contact,
which is conceptually simpler.

When we speak of higher order contact between a surface generated by a conical milling tool and a reference
surface, it is important to note the following: Second order contact, also referred to as \emph{osculation},
means that the surfaces locally penetrate tangentially. Thus, this case is not directly suitable for CNC
machining, but may still be useful for initial estimates of good tool positions. However, third order contact, so-called \emph{hyperosculation},  is locally penetration-free in the very neighborhood of the contact point and therefore very well-suited for CNC machining, in particular
for initialization of optimization algorithms which aim at high-precision machining.

\subsection{Contributions and overview}

Our main contribution is a careful analysis of plane-based higher order contact between cones of
revolution and a given reference surface. This leads to a nonlinear PDE which characterizes exact
envelopes of congruent rotational cones (see Theorem~\ref{cor-cone-envelope}). From a practical perspective, this means that we can
detect the (rare) cases in which a surface can be milled exactly in a single path by flank milling
with an appropriate conical tool, provided that this tool motion is collision free and accessible.
Probably more importantly, a computational approach to locally well fitting tool positions is very helpful
for the initialization of numerical optimization algorithms for high-precision tool motion planning.
On our way towards the characterization of envelopes of moving rotational cones, we discuss other
special types of surfaces as well.

The paper is structured as follows: We discuss relevant previous work in Section~\ref{ssec-previous}. Section \ref{sec-developable} derives a PDE that characterizes the graph of a bivariate function as a developable surface (Theorem~\ref{th-developable}).
Section~\ref{sec-ruled} then extends this characterization to all ruled surfaces (Theorem~\ref{th-ruled}). To extend to envelopes of cones, in Section~\ref{sec-model} we introduce the isotropic model of Laguerre geometry and discuss the contact order between a developable surface and a doubly curved surface, expressed in the space
of planes. Section~\ref{sec-conics} characterizes envelopes of congruent rotational cones in the isotropic model and formulates conditions on second order and third order plane-based contact. This is the basis for proving that a certain PDE characterizes envelopes of
congruent rotational cones (Section~\ref{sec-cones},  Theorem~\ref{cor-cone-envelope}). In Section~\ref{sec:cylinder} we address the
limit case of envelopes of congruent rotational cylinders (Corollary~\ref{cor-offset-ruled}), and for completeness we also discuss envelopes
of spheres in Section~\ref{sec:sphere} (Corollary~\ref{cor-pipe}).  Section~\ref{sec-results} shows examples of hyperosculating cone positions and its application to flank CNC machining. Finally, Section~\ref{sec-conclu} concludes the paper and indicates directions for future research.

\subsection{Previous work}\label{ssec-previous}


\medskip

\textbf{Geometry.} Higher order contact between curves and/or surfaces has been well-studied in the past, see e.g. \cite{kennedy:1994-PlanrCrvs,montaldi:1986-HigherOrderContact,maekawa:1999-SrfSrfIntrs}. It appears, for example, in surface-surface intersection: Using marching methods is straightforward for transversal intersections, however, when the surfaces in question
have higher order contact, the computation of the intersection curve is quite complex \citep{maekawa:1999-SrfSrfIntrs}. Higher order contact between a circle and a surface in Euclidean 3-space is studied in \cite{montaldi:1986-HigherOrderContact}, in particular the existence of  circles with 5-th order contact at the umbilical points of a surface.

Another class of relevant research deals with the approximation of general free-form (NURBS) surfaces by ruled surfaces \cite{hoschek:1998-InterpolRuled,VanSosin-2019-LineCutting}, or even developable surfaces \cite{pottwall:2001,rabinovich:2018-DiscreteGeoNets,solomon:2012-FlexibleDevel,elber:2006-PiecewiseDevel,Tang-2016-Develop}. For simple geometries, the process of approximation can be even \emph{interactive}, while the design of very complex shapes requires many rounds of optimization and is still beyond real-time performance \cite{Tang-2016-Develop}.


With the blossom of modern free-form architecture, another type of research appeared recently. A curved geometry on a large scale requires fine approximation in order to, for example, create panels,  molds for their
production and support structures. This requires segmentation of the whole complex free-form surface into manufacturable patches, while minimizing the cost of the whole manufacturing process \cite{eigensatz:2010-Paneling}. To this end, another promising direction is to use to simple, ideally congruent, curved geometric entities such as circular arcs \cite{Snakes-2013,bo:2011-CAS} in a repetitive manner.

 \begin{figure}[!tb]
\vrule width0pt\hfill
\begin{overpic}[width=0.48\columnwidth]{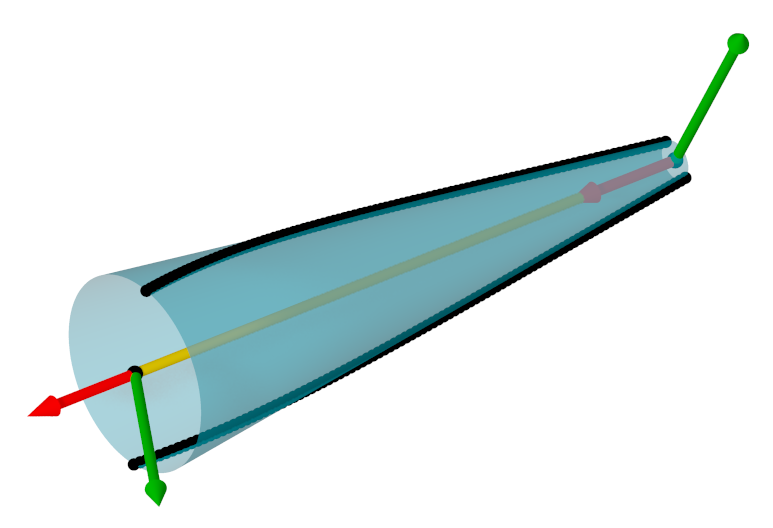}
{\small
  \put(0,-2){(a)}
  \put(55,0){\fcolorbox{gray}{white}{\includegraphics[width=0.19\textwidth]{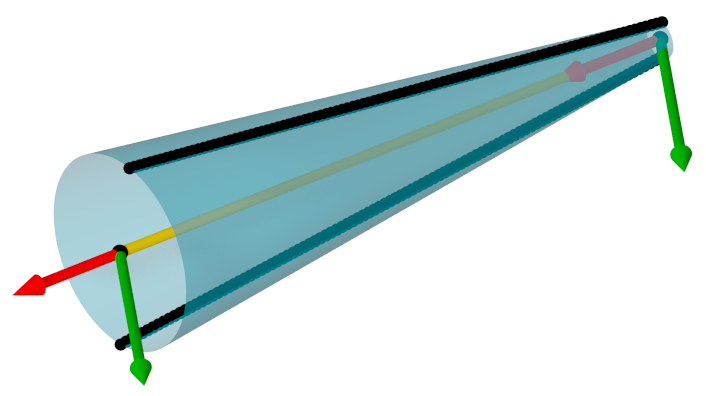}}}
  }
\end{overpic}
\hfill
\begin{overpic}[width=0.48\columnwidth]{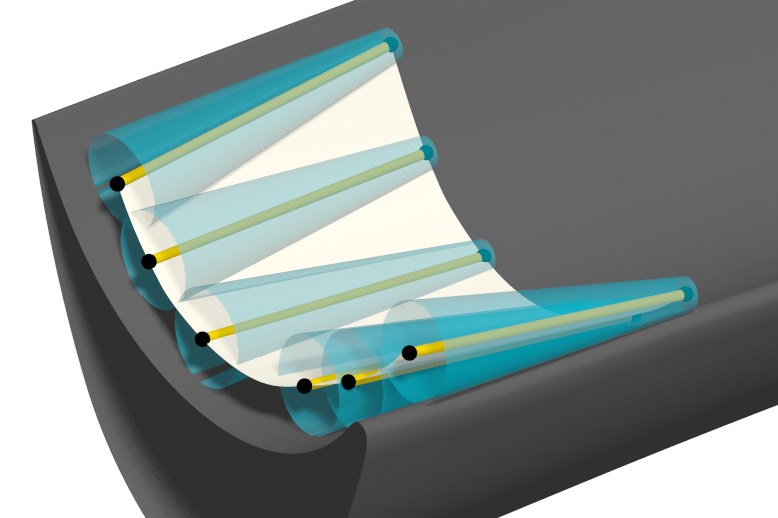}
{\small
  \put(0,-2){(b)}
  }
\end{overpic}
\hfill\vrule width0pt\\
\vspace{-12pt}
\caption{Flank milling with a conical tool. (a) The rotation of the tool about its axis generates a truncated cone (transparent) whose
instantaneous motion is determined by a pair of velocity vectors (green); their projections onto the axis of the cone are two identical vectors (red). The contact curve with the envelope is known as the \emph{characteristic} (black) and is in general position an algebraic curve of degree four. Bottom framed: For special instantaneous motions, such as translation, the characteristic degenerates to a pair of straight lines. (b) In  5-axis flank CNC machining, the goal is to move the tool tangentially to the reference surface (black), that is, to approximate the input surface by an envelope of the moving truncated cone.}
\label{fig:CNC}
\end{figure}

\textbf{CNC machining.} The problem of approximating a general free-form surface by an envelope of a
moving simple object (e.g. a quadric) has been inspired by applications in 5-axis CNC machining.
We refer to the very final stage of 5-axis CNC machining, known as \emph{flank machining}, where the tool, typically a cone or a cylinder, moves tangentially along the to-be-manufactured surface, having a contact with the surface -- theoretically -- along a whole curve, see Fig.~\ref{fig:CNC}.



In the case of 5-axis flank milling with \emph{cylindrical} tools, the tool path-finding problem can be alternatively formulated as approximating the offset surface of the input surface (offset by the radius of the tool) by a set of ruled surfaces. Therefore a lot of literature is devoted to this equivalent formulation, see e.g. \cite{Hsu-2008-Ruled,Gong-2005,li-fm-2006,Menzel-2004-Triple,Redonnet-1998,sprott2008-cmo,Wang-2014,Xu-2017-CuttingForce}
and the fact that a free-form surface can be approximated by ruled surfaces arbitrarily well \cite{Elber-1997}. However, this approximation of a general, doubly-curved surface by ruled surfaces within fine tolerances typically requires an excessive number of patches \cite{Elber-1997}. On the other hand, negatively
curved surfaces can be approximated even by a reasonably small number of smoothly joining ruled surface strips \cite{floery-aag-2012}.

In the case of approximation with \emph{conical} tools, the literature is a lot more sparse.
One can machine a ruled surface perfectly with a cylindrical or conical tool only if the  tangent plane along the ruling is constant, i.e., the surface is developable. For a general (non-developable) ruled surface, an approximation approach is necessary \cite{li-fm-2006}. For general free-form surfaces, an alternative approach is to use an approximation of the surface's distance function and look for directions in which its Hessian vanishes \cite{IniConical-2017}.  Along these 3D directions, the distance from the reference surface changes linearly and therefore provides good initial candidates for the milling axis positions.
%

Another important issue is the accessibility of the surface by a machining tool. A conservative estimate is proposed in the context of 5-axis ball-end milling \cite{Ezair-2018-Collision}. The admissible directions of the tool are encoded using normal bounding cones which enables to quickly find whole volumes in the configuration space that correspond to possible tool paths. As a result, there is no need to compute accessibility for individual cutter contact points which brings significant computational savings.

Real-life manufacturing of free-form surfaces using conical tools is conducted in \cite{Calleja-2018-FlankMillConicalTools}. Using the initialization strategy for flank milling with conical tools introduced in \cite{IniConical-2017}, one quickly finds initial motions (ruled surfaces) of the milling axis and reveals the parts of free-form surfaces that can be efficiently approximated by conical envelopes within very fine machining tolerances. Consequently, high accuracy leads to a reduced machining time as only few sweeps are needed to cover large portions of the surface \cite{Calleja-2018-FlankMillConicalTools}.

Another strong stream of research deals with \emph{curved} tools and especially \emph{barrels} \cite{li-fm-2008,Luo-2016-Barrel,Urbikain-2017-Barrel}. Barrel tools are shown to fit well free-form surfaces, especially in concave regions where the principal curvatures of the tool match their counterparts of the surface. The most recent research focuses on\emph{ custom-shaped tools}. That is, not only the 3D motion of the tool, but also the shape itself are the unknowns in  path-planning \cite{Gong-2009,Yu-2017-OptimizingSize,Yu-2017-ShapeOpt,Zheng-2012-CutterSizeOpt,zhuetal-2012,zhu-2010-goo}. Typically, the initial milling trajectory is a part of the input or is indicated by the user. Recent research focuses on  automatic path initialization for 5-axis flank milling \cite{IniGeneral-2018,IniConical-2017}. For a specific shape of the milling tool (conical or doubly curved), an automatic initialization of the tool's motion can be achieved by integrating the admissible multi-valued vector field that corresponds to directions in which the point-surface distance changes according to the prescribed shape of the milling tool (prescribed by a meridian curve) \cite{IniGeneral-2018}.

On the conceptual level, our research in this paper is closely related to \cite{CurvCatering1-1993,CurvCatering2-1993,MillingCircles-2015}, which concerns research on 5-axis  flat-end milling with cylindrical tools, where the bottom circle is posed in third order contact (hyperosculation) with the reference surface. In this work, however, we have to deal with higher order contact in the space of planes, i.e.,
we look for hyperosculation between a special conic and a curved surface in the isotropic model of Laguerre geometry, and not in Euclidean space.


\section{Developable surfaces}\label{sec-developable}

Developable surfaces can be mapped isometrically into the plane and therefore appear when working with thin sheets of materials which are much more easily bent than stretched, such as paper, certain plastics, and sheet metal. We first treat this well-known class of surfaces and in this way introduce to our approach at hand of a well-known case.

There are several properties of developable surfaces which may serve as equivalent definitions. For instance, they are locally envelopes of one-parameter families of planes. Hence, we use the following characterization of a developable surface: a tangent plane touches the surface along a straight line segment (\emph{ruling}); see Fig.~\ref{fig:DevelSrf}.

We derive a well-known PDE for developable surfaces (see Theorem~\ref{th-developable} below) which is equivalent to
expressing vanishing Gaussian curvature \cite[Example~5 in \S3.3]{docarmo:1976:DG}.  We take the surface to be the graph of a smooth function $f$. Then the required PDE is obtained by differentiation along a ruling. Conversely, given $f$ satisfying the PDE, the rulings are reconstructed as follows. First we get the ruling directions from the derivatives of $f$. Then we integrate the resulting vector field. Finally, we use the PDE to ensure that the integral curves are actually line segments and the tangent planes are constant along the segments. Such reconstruction works for a generic point on the surface.

\subsection*{Derivation of the PDE}
Assume that a $C^3$ function $f(x,y)$ is defined in an open disk and a plane $z=ax+by+c$ touches the graph along a line segment $(x+ut, y+vt, z+wt)$, where $t$ runs through a segment $(-\epsilon,\epsilon)$ and $a,b,c,u,v,w\in\mathbb{R}$ are fixed. Then $f_x(x+ut,y+vt)=a$ and $f_y(x+ut,y+vt)=b$ identically. Differentiating with respect to $t$ we get $f_{xx}u+f_{xy}v=f_{xy}u+f_{yy}v=0$. Since $(u,v)\ne (0,0)$ it follows that
$$f_{xx}f_{yy}- {f_{xy}}^2= 0,$$
which is nothing but vanishing of the Gaussian curvature $K=\frac{f_{xx}f_{yy}-{f_{xy}}^2}{(1+f_x^2+f_y^2)^2}$; see e.g.,
\cite[Example~5 in \S3.3]{docarmo:1976:DG}.


 \begin{figure}[!tb]
\vrule width0pt\hfill
\begin{overpic}[width=0.48\columnwidth]{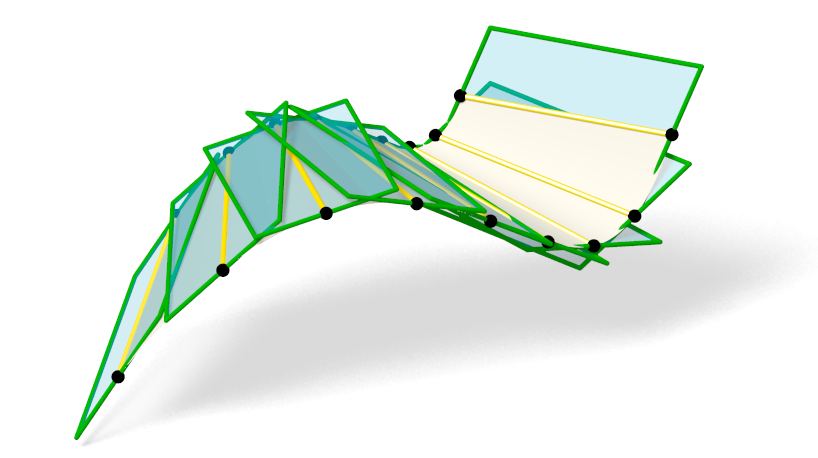}
{\small
  \put(0,-2){(a)}
  }
\end{overpic}
\hfill
\begin{overpic}[width=0.48\columnwidth]{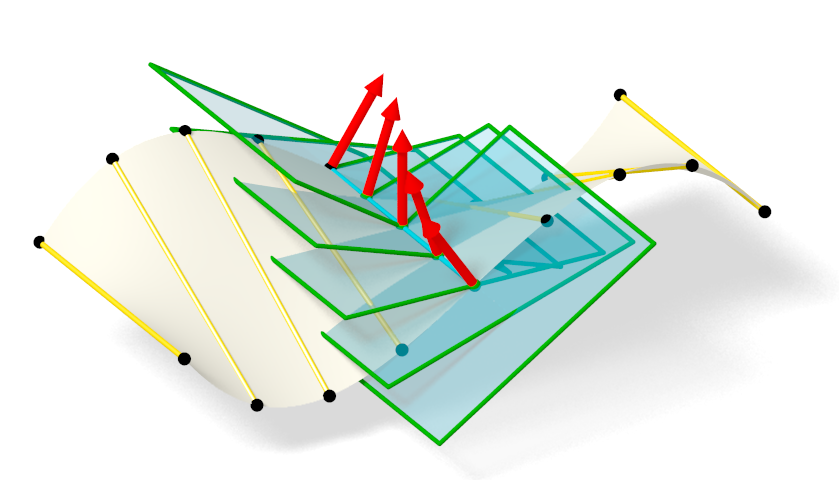}
{\small
  \put(0,-2){(b)}
  }
\end{overpic}
\hfill\vrule width0pt\\
\vspace{-12pt}
\caption{(a) A developable ruled surface is an envelope of a one-parameter family of planes. Every plane touches the surface along the whole ruling. In contrast, (b) tangent planes vary along a generic ruling of a non-developable ruled surface.}
\label{fig:DevelSrf}
\end{figure}


\subsection*{Reconstruction of the rulings}
To show that the resulting PDE implies that the surface is developable, first consider a point where $f_{xx}f_{yy}\ne 0$. By the continuity, $f_{xx}f_{yy}\ne 0$ in a neighborhood of the point.
%
Consider the vector field $(u,v)=(f_{yy},-f_{xy})/\sqrt{f_{yy}^2+f_{xy}^2}$ in this neighborhood. It is $C^1$, if $f$ is $C^3$. Integrate the field: by the Picard-Lindel\"of theorem on the existence of integral curves for $C^1$ fields, for some $\epsilon>0$ there is a regular curve $(x(t),y(t))$ with $\dot{x}(t)=u(x(t),y(t))$ and $\dot{y}(t)=v(x(t),y(t))$ for each $t\in(-\epsilon,\epsilon)$.

Let us prove that $(x(t),y(t))$ is actually a straight line segment. Hereafter all equations are understood as holding for each $t\in(-\epsilon,\epsilon)$, and all the functions $u=u(x(t),y(t)), v=v(x(t),y(t)), f=f(x(t),y(t))$ and their derivatives are evaluated at the point $(x(t),y(t))$.
Differentiating
\begin{equation}\label{eq-auxiliary0}
f_{xx}u+{f_{xy}}v=\frac{f_{xx}f_{yy}- {f_{xy}}^2}{\sqrt{f_{yy}^2+f_{xy}^2}}=0
\end{equation}
with respect to $t$ and substituting $\dot{x}=u$, $\dot{y}=v$, we get
\begin{equation}\label{eq-auxiliary1}
f_{xxx}\dot{x}u+f_{xxy}(\dot{y}u+\dot{x}v)+f_{xyy}\dot{y}v
+f_{xx}\dot{u}+{f_{xy}}\dot{v}=
f_{xxx}u^2+2f_{xxy}uv+f_{xyy}v^2
+f_{xx}\dot{u}+{f_{xy}}\dot{v}
=0.
\end{equation}
Let us simplify the resulting expression further. Differentiating the equation $f_{xx}f_{yy}- {f_{xy}}^2= 0$ with respect to $x$, multiplying by $f_{yy}/(f_{yy}^2+f_{xy}^2)$, and using the equation again, we get
\begin{equation*}\label{eq-auxiliary}
(f_{xxx}f_{yy}-2f_{xxy}f_{xy}+f_{xyy}f_{xx})\frac{f_{yy}}{f_{yy}^2+f_{xy}^2}
=f_{xxx}u^2+2f_{xxy}uv+f_{xyy}v^2=0.
\end{equation*}
Subtracting the resulting equation from~\eqref{eq-auxiliary1} we arrive at
$$
f_{xx}\dot{u}+{f_{xy}}\dot{v}=0.
$$
Comparing with~\eqref{eq-auxiliary0} we get $(\dot{u},\dot{v})\parallel(u,v)$ because $f_{xx}\ne 0$. Since $u^2+v^2=1$ it follows that $u\dot{u}+v\dot{v}=0$, hence $\dot{u}=\dot{v}=0$ and $(x(t),y(t))$ is a line segment. The tangent plane is constant along the segment because $\frac{d}{dt}f_x(x(t),y(t))=f_{xx}u+f_{xy}v=0$ and $\frac{d}{dt}f_y(x(t),y(t))=f_{xy}u+f_{yy}v=0$.
In particular, the plane is tangent to the graph of $f$ along the segment $(x+ut,y+vt,f(x,y)+(f_xu+f_yv)t)$,
where this time all $x,y,u,v,f_x,f_y$ are evaluated at $t=0$.

\subsection*{Technical conventions}
Now we address a technical issue: we need to consider separately the subsets where one of the derivatives $f_{xx}$ or $f_{yy}$ vanishes. In (each connected component of) their interior, the graph of the function $f$ is a cylinder (not necessarily of revolution) and thus is trivially developable. Indeed, if, say, $f_{xx}=0$ inside some square (with the sides parallel to the $x$- and $y$-axes), then by the PDE we get $f_{xy}=0$, hence $f_x=\mathrm{const}$ and $f(x)=ax+b(y)$ for some $a,b(y)$. Then the points on each segment $y=\mathrm{const}$ have a common tangent plane.

Theoretically the boundaries of the subsets where $f_{xx}=0$ or $f_{yy}=0$ can be complicated fractals (see also \cite[Example~1 in \S5.8]{docarmo:1976:DG}):

\begin{example} Take a Cantor set of positive Lebesque measure. Take a $C^\infty$ function $g(x)$ vanishing on the set and positive outside it. Take a $C^\infty$ function $f(x,y)$ not depending on $y$ such that $f_{xx}(x,y)=g(x)$. Then $f_{xx}f_{yy}- {f_{xy}}^2= 0$, and the boundary of the set $f_{xx}=0$ is formed by all $(x,y)$ with $x$ in the Cantor set, hence it is an uncountable union of segments and has positive measure.
\end{example}


Practically the boundary (if nonempty at all) is a curve on the surface, hence ``negligible'' (although still sensible because our algorithm may become unstable near it). To avoid too much technicalities while keeping our work mathematically correct, we prefer to limit ourselves to ``generic'' points on a surface.


\begin{definition*} A \emph{negligible} subset is a countable union of subsets such that the closure of each one has no interior points. We say that an assertion holds at a \emph{generic} point, if it holds outside a negligible set.
\end{definition*}

For instance, the boundary of the zero set of a continuous function is always negligible. On the other hand, whatever small disc in the plane is not negligible (this is the Baire category theorem).

The above PDE derivation remains true even if the assumption on the tangent plane is imposed at a \emph{generic} point of the surface rather than \emph{each} point. Indeed, then we conclude that the PDE holds at a generic point $(x,y)$; but since the left-hand side is continuous, the PDE must hold at \emph{each} point as well.

We have arrived at the following theorem.

\begin{theorem}[characterization of developable surfaces] \label{th-developable} For a $C^3$ function $f\colon D\to \mathbb{R}$ defined in an open disk $D\subset\mathbb{R}^2$ the following $2$ conditions are equivalent:
\begin{enumerate}
  \item\label{item2-th-developable}
  The tangent plane to the graph of $f$ at a generic point is   tangent to the graph along at least one straight line segment passing through the point.
  \item\label{item3-th-developable}
  For each $(x,y)\in D$ we have $f_{xx}f_{yy}- {f_{xy}}^2= 0$, i.e., vanishing Gaussian curvature.
\end{enumerate}
\end{theorem}

\begin{remark} Theorem~\ref{th-developable} remains true, if ``a generic point'' is replaced by ``each point''. The proof is obtained by the above argument plus \cite[Proposition~3 in \S5.8]{docarmo:1976:DG} due to W.S.~Massey. The proof of the additional proposition is more complicated although still elementary.
\end{remark}

\section{Ruled surfaces}\label{sec-ruled}

Developable surfaces are (composed of) special ruled surfaces. The converse is not true. At a generic, so-called non-torsal ruling of a ruled surface, the tangent plane is not constant (Fig.~\ref{fig:DevelSrf}(b); for a detailed discussion, see \cite{pottwall:2001}). Hence, we treat ruled surfaces separately as they appear as limits of surfaces enveloped by a family of congruent rotational cones when the opening angle tends to zero whereas the vertices stay fixed.
Again, we derive a PDE characterizing ruled surfaces. It is by far less known than the one for developable surfaces. The classical origin is found in affine differential geometry (see Blaschke \cite{blaschke:1923}, cf.~\cite{monge-1780}), where ruled surfaces are characterized by the vanishing of a 3rd order differential invariant, called {\it Pick's invariant}. To our knowledge, the resulting PDE was first written explicitly by R.~Bryant recently \cite{briant}. All that is equivalent to the result (Theorem~\ref{th-ruled}) given below. Our approach is elementary and does not require knowledge in affine differential geometry.

The PDE for ruled surfaces is found in similar way as the one for developable surfaces in Section~\ref{sec-developable}. Again we consider the graph of a smooth function $f$. Differentiation along a ruling gives a system of algebraic equations on the ruling direction. Taking the resultant, we get a single PDE on $f$ (plus an inequality guaranteeing that the solutions are real). Conversely, given $f$ satisfying the PDE and the inequality, the rulings are reconstructed as follows. First we pick up a suitable normalized solution of our system at each point to get a smooth vector field (directions of ruling projections). Then we prove that the integral curves of the field are straight line segments and 
the restriction of $f$ to these segments is linear.


\subsection*{Derivation of the PDE}
Assume that the segment $(x+ut, y+vt, z+wt)$, where $t$ runs through $(-\varepsilon,\varepsilon)$ and $u,v,w\in\mathbb{R}$ are fixed, is contained in the graph of a $C^3$ function $f(x,y)$.
Then $z+wt=f(x+ut,y+vt)$ identically. Differentiating $3$ times with respect to $t$ consecutively, we get
\begin{equation}\label{eq-th-ruled}
         \begin{cases}
          f_{xx} u^2 + 2f_{xy}uv + f_{yy}v^2=0, & \\
          f_{xxx}u^3+3f_{xxy}u^2v+3f_{xyy}uv^2+f_{yyy}v^3=0. &
        \end{cases}
      \end{equation}


The solvability of the system~\eqref{eq-th-ruled} is analysed directly. The two equations~\eqref{eq-th-ruled} have a common solution $(u,v)$, if and only if the resultant of the left-hand-side polynomials vanishes:
     \begin{multline}\label{eq-th-ruled-resultant}
        {f_{yy}}^3{f_{xxx}}^2+6{f_{yy}}{f_{xxx}}{f_{yyy}}{f_{xy}}{f_{xx}}
        -6{f_{yy}}^2{f_{xxx}}{f_{xyy}}{f_{xx}}-6{f_{yyy}}{f_{xy}}{f_{xx}}^2{f_{xyy}} \\
        +9{f_{yy}}{f_{xyy}}^2{f_{xx}}^2-6{f_{xy}}{f_{yy}}^2{f_{xxy}}{f_{xxx}}
        +12{f_{xy}}^2{f_{xxy}}{f_{yyy}}{f_{xx}}-18{f_{xy}}{f_{yy}}{f_{xxy}}{f_{xyy}}{f_{xx}}
        \\
        +12{f_{yy}}{f_{xyy}}{f_{xy}}^2{f_{xxx}}-8{f_{yyy}}{f_{xy}}^3{f_{xxx}}
        +9{f_{xx}}{f_{yy}}^2{f_{xxy}}^2-6{f_{yy}}{f_{xxy}}{f_{yyy}}{f_{xx}}^2
        +{f_{yyy}}^2{f_{xx}}^3 = 0.
     \end{multline}
The first equation of~\eqref{eq-th-ruled} has a real solution (and moreover all solutions are proportional to real ones), if and only if $f_{xx}f_{yy}- {f_{xy}}^2\le 0$, i.e., the Gaussian curvature $K$ of the surface is non-positive. By the property of the resultant, \eqref{eq-th-ruled} has a real solution $(u,v)$, if and only if we have~\eqref{eq-th-ruled-resultant} and $f_{xx}f_{yy}- {f_{xy}}^2\le 0$.

Geometrically, the first equation of~\eqref{eq-th-ruled} expresses that the segment is an asymptotic direction (direction of vanishing normal curvature; see \cite{docarmo:1976:DG}). If both equations in~\eqref{eq-th-ruled} are satisfied, the line in direction $(u,v,w)$ (with $w=f_xu+f_yv$) has 3rd order contact with the surface (cf.~Definition~\ref{def-contact} below). 



In particular, \eqref{eq-th-ruled} has a real solution $(u,v)$ at \emph{each} point $(x,y)$, even if we assume that the surface contains a line segment through a \emph{generic} point only. 
Indeed, then the PDE and the inequality hold at a generic point $(x,y)$; but since their left-hand sides are continuous, they must hold everywhere.

\subsection*{Yet another technical issue}

To prove that the resulting PDE and the inequality imply that the surface is ruled, we address yet another technical issue. Assume that \eqref{eq-th-ruled} has a nonzero real solution $(u,v)$ at each point $(x,y)$. We would like to pick up a nonzero solution $(u(x,y),v(x,y))$ \emph{smoothly} ($C^1$) depending on the point $(x,y)$. This is not possible in general: for instance, the solutions of the system
$$
\begin{cases}
  u^2-v^2=0, \\
  xu^3-|x|v^3=0;
\end{cases}
$$
are proportional to $(1,1)$ for $x\ge 0$ and to $(1,-1)$ for $x\le 0$. Thus we have to restrict to a smaller domain as follows.

Since the first equation of~\eqref{eq-th-ruled} has a real solution, it follows that $f_{xx}f_{yy}- {f_{xy}}^2\le 0$.
In the interior of the subset where $f_{xx}f_{yy}- {f_{xy}}^2=0$, the surface is developable, hence ruled, by Theorem~\ref{th-developable}. Drop the negligible boundary of the subset and further restrict to the subset where $f_{xx}f_{yy}-{f_{xy}}^2<0$. 

Consider the auxiliary system consisting of the first equation of~\eqref{eq-th-ruled} and the equation $u^2+v^2=1$. The former
is quadratic with the discriminant $f_{xx}f_{yy}- {f_{xy}}^2<0$,
hence defines a pair of lines passing through the origin in the $(u,v)$ plane. The latter equation defines a circle transversal to the lines. Hence the system has exactly $4$ solutions, none of which are multiple. By the implicit function theorem, in a sufficiently small neighborhood of any point $(x_0,y_0)$ the solutions form $4$ smooth branches $(u_k(x,y),v_k(x,y))$, where $k=1,2,3,4$.

For each $k=1,2,3,4$ consider the closed subset where $(u_k(x,y),v_k(x,y))$ satisfies the second equation of~\eqref{eq-th-ruled} as well. These $4$ subsets cover
the whole neighborhood in question and have negligible boundary. Thus a generic point belongs to the interior of one of these subsets.

\subsection*{Reconstruction of the rulings}

We have proved that if \eqref{eq-th-ruled} has a nonzero real solution $(u,v)$ and $f_{xx}f_{yy}-{f_{xy}}^2\ne 0$, then in a neighborhood of a generic point there is a solution $(u(x,y),v(x,y))$ depending smoothly ($C^1$) on $(x,y)$ such that $u(x,y)^2+v(x,y)^2=1$.

Then by the Picard-Lindel\"of theorem for some $\epsilon>0$ there is a regular curve $(x(t),y(t))$ with $\dot{x}(t)=u(x(t),y(t))$ and $\dot{y}(t)=v(x(t),y(t))$ for each $t\in(-\epsilon,\epsilon)$. Let us prove that $(x(t),y(t))$ is a straight line segment and $f(x(t),y(t))$ is linear.

The left-hand side of the first equation of~\eqref{eq-th-ruled} is a function on the curve $(x(t),y(t))$ vanishing identically. Differentiating the function with respect to $t$ we get
$$
f_{xxx}\dot{x}u^2+f_{xxy}(\dot{y}u^2+2\dot{x}uv)
+f_{xyy}(2\dot{y}uv+\dot{x}v^2)+f_{yyy}\dot{y}v^2
+2f_{xx}u\dot{u}+2f_{xy}(u\dot{v}+v\dot{u})+2f_{yy}v\dot{v}
=0.
$$
Substituting $\dot{x}=u$, $\dot{y}=v$, and subtracting the second equation of~\eqref{eq-th-ruled}, we get
$$
f_{xx}u\dot{u}+f_{xy}(u\dot{v}+v\dot{u})+f_{yy}v\dot{v}
=0.
$$
Then by the first equation of~\eqref{eq-th-ruled} both $(\dot{u},\dot{v})$ and $(u,v)$ are orthogonal to the vector $(f_{xx}u+f_{xy}v,f_{xy}u+f_{yy}v)$. The latter is nonzero because $f_{xx}f_{yy}- {f_{xy}}^2\ne 0$ and $u^2+v^2=1$. Hence $(\dot{u},\dot{v})\parallel(u,v)$.
Since $u^2+v^2=1$ it follows that $\dot{u}=\dot{v}=0$ and $(x(t),y(t))$ is a line segment. The restriction of $f$ to the segment is linear because $\frac{d^2}{dt^2}f(x(t),y(t))=0$ by the first equation of~\eqref{eq-th-ruled}.

We have arrived at the following theorem.

\begin{theorem}[characterization of ruled surfaces] \label{th-ruled} For a $C^3$ function $f\colon D\to \mathbb{R}$ defined in an open disk $D\subset\mathbb{R}^2$ the following $3$ conditions are equivalent:
\begin{enumerate}
  \item\label{item1-th-ruled}
  Through a generic point of the graph of $f$ there passes a line segment completely contained in the graph.
  \item\label{item2-th-ruled}
  For each $(x,y)\in D$, the two equations~\eqref{eq-th-ruled} have a common nonzero real solution $(u,v)$.
  \item\label{item3-th-ruled}
  For each $(x,y)\in D$ we have~\eqref{eq-th-ruled-resultant} and $f_{xx}f_{yy}- {f_{xy}}^2\le 0$ (i.e., Gaussian curvature is nonpositive).
\end{enumerate}
\end{theorem}

\begin{remark}
In case of strictly negative Gaussian curvature, our argument shows that the graph contains a \emph{continuous} family of line segments (and even an analytic family, if $f$ is analytic, cf.~\cite[Proof of Corollary~3]{Skopenkov-Pottmann-Grohs:2012}).
\end{remark}


\section{Surfaces enveloped by a family of rotational cones, using a point model of the space of planes}
\label{sec-model}

Now we come to the main topic of the paper: how to characterize surfaces enveloped by a one-parametric family of congruent cones? To minimize technicalities, we consider surfaces tangent to cones along curves rather than arbitrary envelopes of cones, and exclude certain positions of these curves.
In this section we reduce the problem to the characterization of surfaces containing a special conic through each point, which is tractable by the methods already discussed.

\subsection*{Motivation}

The motivation for using a plane-based approach is the following. A cone has just a one-parameter family
of tangent planes $T(u)$. Moving the cone, seen as set of its tangent planes, under a generic smooth one-parameter motion,  \emph{we obtain a two-parameter family of planes $T(u,v)$. These are precisely the tangent planes of the envelope!}

One can convert the resulting (plane) representation of the envelope into its dual (point) variant by computing the intersection points
\begin{equation} \label{dual-to-point}
    \mathbf{r}(u,v) = T(u,v) \cap T_u(u,v) \cap T_v(u,v),
\end{equation}
where $T_u(u,v)$ and $T_v(u,v)$ are the planes with the equations obtained from the equation of $T(u,v)$ by taking partial derivatives. That is, if $T(u,v)$ has the equation
 $$n_1(u,v)x+n_2(u,v)y+n_3(u,v)z+h(u,v)=0,$$
 then $T_u(u,v)$ is given by
 $$\frac{\partial n_1(u,v)}{\partial u}x+\frac{\partial n_2(u,v)}{\partial u}y+\frac{\partial n_3(u,v)}{\partial u}z+\frac{\partial h(u,v)}{\partial u}=0. $$
 This equation will not degenerate, as the intersections $T(u,v) \cap T_u(u,v)$ are the rulings of the cone. However, $T_v(u,v)$
 may degenerate. Even all four partial derivatives with respect to $v$ may vanish simultaneously at particular points.
 Also, even if $T_v(u,v)$ is a well defined plane,  the intersection (\ref{dual-to-point}) may be at infinity or be an entire straight line. For our purposes, it is
 not important to discuss all these cases and the corresponding properties of the generating motion. This is why we talked about
 a \emph{generic} motion, which we want to define as one where (\ref{dual-to-point}) is always a well-defined point in $\mathbb{R}^3$ smoothly depending on $u,v$.



 A few more remarks are in place:  Note that we consider the whole unbounded moving cone and the possibly unbounded
envelope. Also note that the envelope may consist of several parts and may have self-intersections. For example, when a rotational cylinder of radius $r$ moves in a way such that its axis
remains tangent to a generic space curve $c$, the envelope consists of two offset surfaces of the tangent developable of the curve $c$ and a pipe surface (the envelope of
spheres of radius $r$, centered at $c$). By the way, the latter part of the envelope is useless for the CNC machining application we have in mind. We prefer to avoid envelopes in the precise statements of our results because this notion has slightly different definitions in the literature. (Sometimes this even leads to confusion: e.g., osculating circles of a generic curve are nested but all tangent to the curve; their envelope is the curve itself or empty depending on the choice of definition. In view of that notice that \cite[Lemma~7]{Skopenkov-Pottmann-Grohs:2012} remains true for nested circles and should be applied in case~(3) of the proof of Theorem~4 there.)

Anyway, converting the plane representation of the envelope into the point one is a postprocessing step and is not necessary for a characterization of these envelopes when we work in the space of planes.

\subsection*{Definition of the point model}

Since geometric processing is easier in terms of points rather than planes, we apply a map that transforms planes to points and use a certain duality between plane and point coordinates. As we work with rotational cones, we use a transformation which allows us easily to recognize these cones in the point model.
The right setting is that of {\it Laguerre geometry}\footnote{Another well-known assignment of points to planes is \emph{polarity} with respect to the unit sphere. But it does not work that well because leads to ``linear'' functions on the sphere rather than on the plane, which are hard to deal with.},
the geometry of oriented planes \cite{blaschke:1929, cecil:1992}. Laguerre geometry has already been useful in various applications
in CAGD, see \cite{DOHM2009111,KRASAUSKAS2000101,krasauskas-zube,peternell-1998-lgaro,pottmann-1998-alcagd}.
We try to give a concise, precise, and self-contained introduction to the subject; this is an update of \cite[\S2.3]{Skopenkov-Pottmann-Grohs:2012}.

We introduce the following coordinates for planes in space. Let an oriented plane $P$ be given by the equation $n_1x+n_2y+n_3z+h=0$, where $(n_1,n_2,n_3)\ne(0,0,-1)$ is the oriented unit normal to the plane and $|h|$ is the distance from the origin. The desired coordinates of the plane $P$ is the triple
\begin{equation}\label{eq-isotropic-model}
\left(\frac{n_1}{n_3+1},\frac{n_2}{n_3+1},\frac{h}{n_3+1}\right). 
\end{equation}
For the geometric considerations which lead to such coordinates, 
we refer to \cite{pottmann-1998-alcagd,pottmann-2009-lms,pottmann-2007-ds}.
To think geometrically, denote by $P^i$ the point with these coordinates, see Fig.~\ref{fig:IsotropicMap}. This correspondence between planes and points is called the \emph{isotropic model of Laguerre geometry}; see \cite{pottmann-1998-alcagd,pottmann-2009-lms,pottmann-2007-ds}.
The simple non-Euclidean geometry in the point model,
known as {\it isotropic geometry}, is treated in detail in \cite{Sachs:1990}.

 \begin{figure}[!tb]
\vrule width0pt
\begin{overpic}[width=0.32\columnwidth]{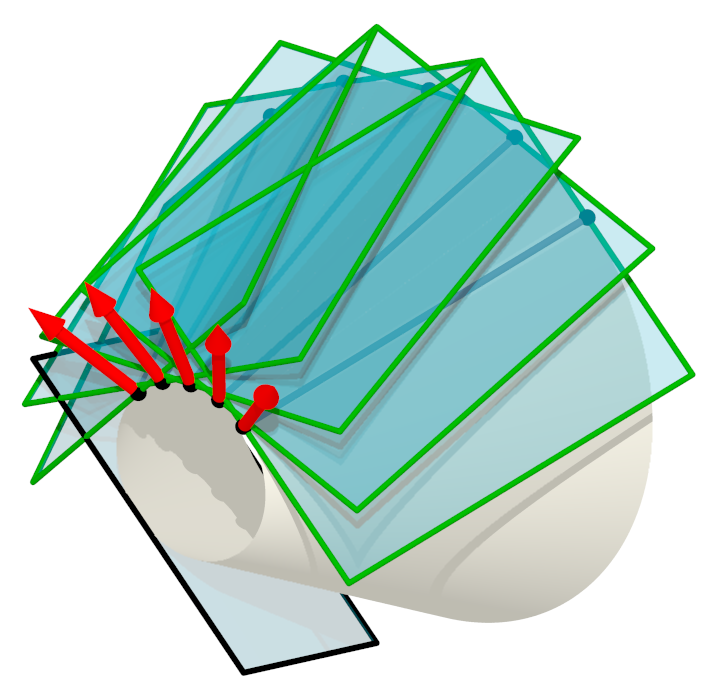}
{\small
	\put(0,80){\fcolorbox{gray}{white}{Euclidean space}}
	\put(70,65){\fcolorbox{gray}{white}{\includegraphics[width=0.09\textwidth]{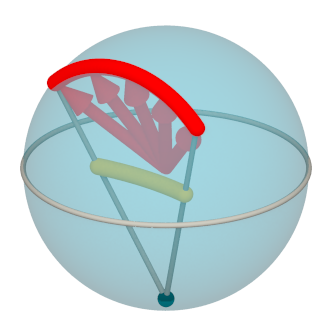}}}
	\put(38,0){\fcolorbox{gray}{white}{\tiny $n_1x+n_2y+n_3z+h=0$}}
  	\put(0,-2){(a)}
}
\end{overpic}
\hfill
\begin{overpic}[width=0.32\columnwidth]{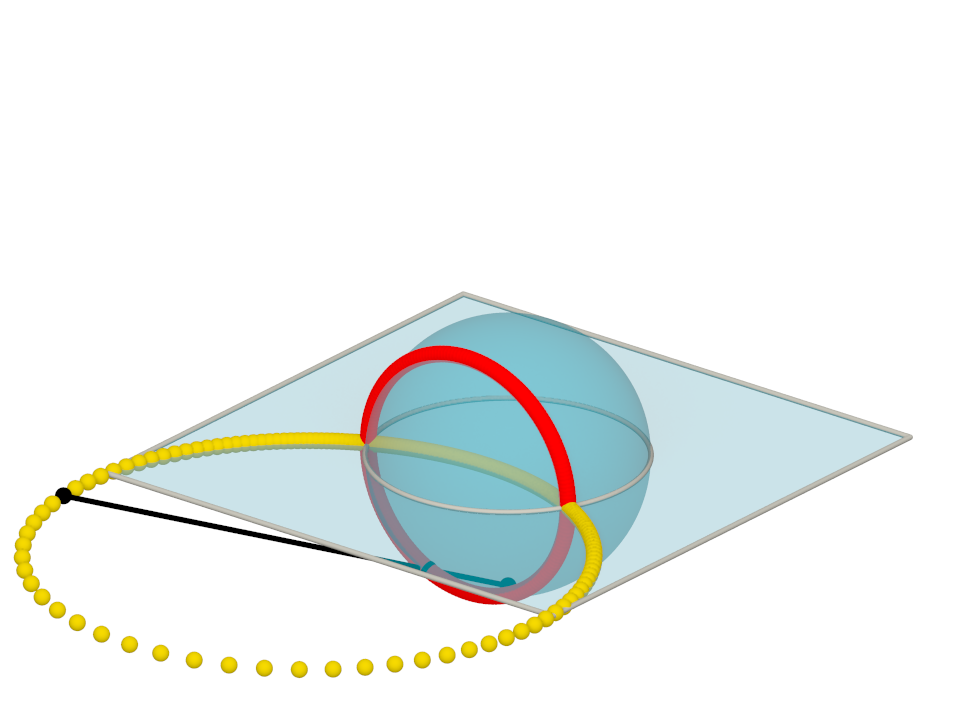}
{\small
	\put(0,-2){(b)}
	\put(83,20){$z=0$}
  }
\end{overpic}
\hfill
\begin{overpic}[width=0.32\columnwidth]{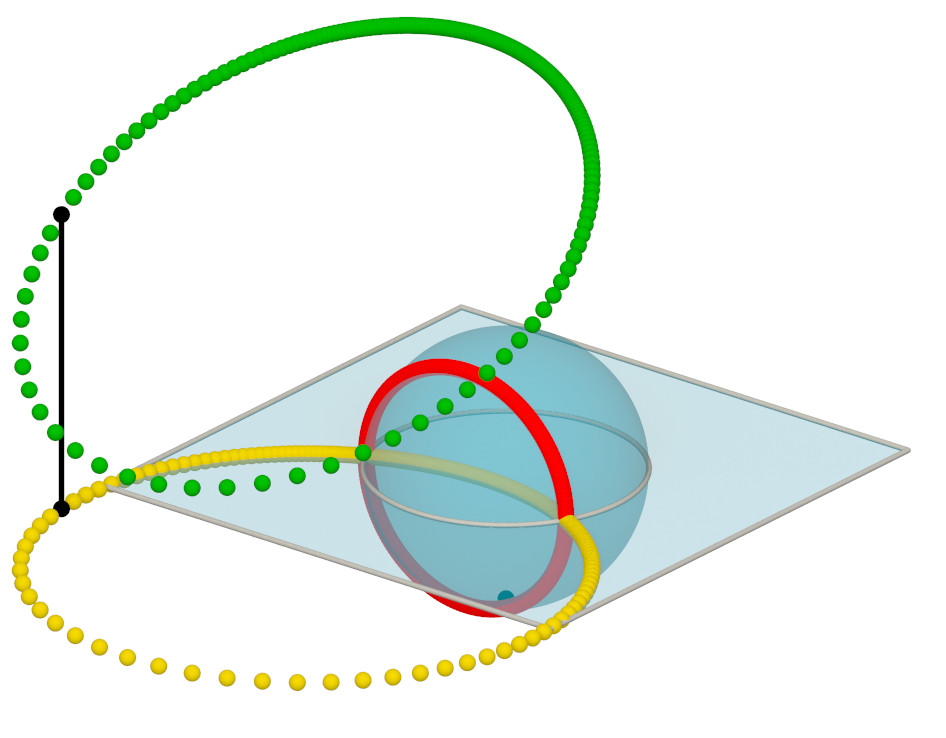}
{\small
\put(-40,80){\fcolorbox{gray}{white}{Isotropic space}}
  \put(0,-2){(c)}
  \put(7,53){$\left(\frac{n_1}{n_3+1},\frac{n_2}{n_3+1},\frac{h}{n_3+1}\right)$}
  }
\end{overpic}
\hfill\vrule width0pt\\
\vspace{-12pt}
\caption{In the isotropic model of Laguerre geometry, planes appear as points and the tangent planes of a rotational cone are seen as special conics (isotropic circles): (a) A cone is considered as a one-parameter family of oriented planes and its normals (red) define a circle (only an arc is shown) on the Gaussian sphere (top-framed). (b) The Gaussian image (red circle) is projected from the south pole $(0,0,-1)$ to the $z=0$ plane to define the ``top view'' of the isotropic image (yellow). (d) The isotropic image of the cone (green conic), see Eq.~\eqref{eq-isotropic-model}.}
\label{fig:IsotropicMap}
\end{figure}
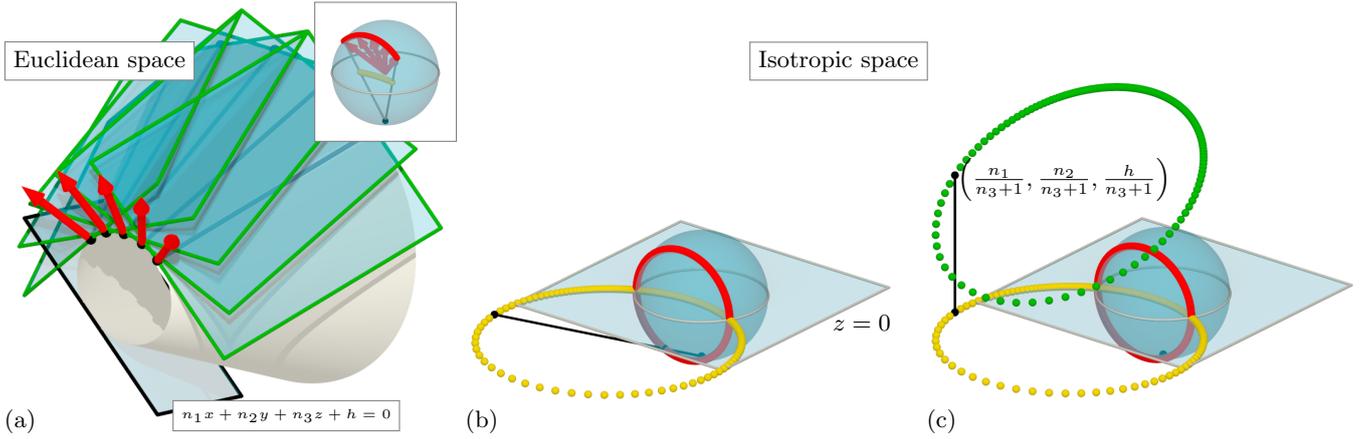

To map an oriented surface $\Phi$ to the isotropic model, we consider the set $\Phi^i$ of points $P^i$, where $P$ runs through all oriented tangent planes to $\Phi$ with the oriented normals distinct from $(0,0,-1)$. Hereafter by an \emph{oriented surface} we mean the image of a proper injective $C^2$ map of an open disk --- or more generally of a smooth $2$-manifold --- into $\mathbb{R}^3$ with nondegenerate differential at each point, equipped with oriented unit normals continuously depending on the point.
For example, a sphere with center $(m_1,m_2,m_3)$, radius $R$, and inwards oriented normals is mapped to the rotational paraboloid (possibly degenerating to a plane),
\begin{equation}
  z=\frac{R+m_3}{2}(x^2+y^2)-m_1x-m_2y+\frac{R-m_3}{2} \label{iso-sphere}.
\end{equation}

The projection $(x,y,z) \mapsto (x,y,0)$ of space onto the $xy$-plane is called \emph{top view}. The top view of $\Phi^i$ is actually the stereographic projection of the Gaussian spherical image of $\Phi$ from the point $(0,0,-1)$ to the $xy$-plane, see
Fig.~\ref{fig:IsotropicMap}(b). In particular, if the Gaussian curvature of $\Phi$ does not vanish, then $\Phi^i$ is locally a graph of a function.

By a \emph{cone} we mean an {\it oriented cone of revolution}\footnote{To be precise, we exclude the vertex to get a smooth surface.}.
The \emph{opening angle} $\theta$ of a cone is the angle between the axis and a ruling. A cone, viewed as the common tangent planes of two oriented spheres, is mapped to the common points of two paraboloids of form~\eqref{iso-sphere}, i.e. a conic with the top view being a circle (or a parabola with the top view being a line). Such a conic is called a \emph{circle in isotropic geometry}, or \emph{isotropic circle}.

The shape of the top view can be obtained algebraically by eliminating $z$ from the system of two equations of form~\eqref{iso-sphere}. But geometry gives more insight: the top view is the stereographic projection of the Gaussian spherical image of the cone, i.e., the projection of a circle of intrinsic radius $\pi/2-\theta$ on the unit sphere; see Fig.~\ref{fig:IsotropicMap}. This leads to the following key observations.

\begin{proposition} \label{prop-cone-translation} For a cone $C$ with the opening angle $\theta$ such that all the oriented unit normals are distinct from $(0,0,-1)$ the set $C^i$ is a conic satisfying the following condition:

$(\Theta)$ the top view of the conic is the stereographic projection of a circle of intrinsic radius $\pi/2-\theta$ in the unit sphere (not passing through the projection center $(0,0,-1)$).
\end{proposition}

\begin{proposition} \label{prop-translation} Let $\Phi$ be an oriented  surface in $\mathbb{R}^3$ with nowhere vanishing Gaussian curvature and the oriented unit normals distinct from $(0,0,-1)$. Then the following two conditions are equivalent:
	\begin{itemize}
		\item through each point of $\Phi$ there passes an oriented cone which is tangent to $\Phi$ along a continuous curve  containing the point (not a ruling because the Gaussian curvature of $\Phi$ does not vanish), has the opening angle $\theta$, and has no oriented unit normals of the form $(0,0,-1)$;
	    \item through each point of $\Phi^i$ there passes an arc of a conic contained in $\Phi^i$ and satisfying condition~$(\Theta)$.
	\end{itemize}
\end{proposition}


Practically the pieces of $\Phi$ where the Gaussian curvature vanishes are developable, hence trivially millable by a conical tool (possibly except the boundary of the set where the mean curvature vanishes).
Thus in what follows we assume that the design surface $\Phi$ satisfies the following condition (likewise, this should be assumed throughout
\cite[\S2.3]{Skopenkov-Pottmann-Grohs:2012}).

\smallskip \noindent
{\bf Condition (*)} $\Phi$ is an oriented surface in $\mathbb{R}^3$ with nowhere vanishing Gaussian curvature such that all the oriented unit normals are distinct from $(0,0,-1)$, and $\Phi^i$ is the graph of a $C^4$ function $f\colon D\to \mathbb{R}$ in a disk $D\subset \mathbb{R}^2$.
\smallskip

This reduces the characterization of surfaces enveloped by a family of cones to the characterization of surfaces (actually graphs of functions) containing a special conic through each point.

\subsection*{Contact order in the space of planes}
\label{ssec-contact-order}

Recall that the derivation of the PDE for ruled surfaces has been based on expressing 3rd order contact between a straight line
and a surface. We take a similar approach in the isotropic model of Laguerre geometry by looking at higher order contact between
an isotropic circle and a surface. We now informally discuss the geometric meaning in the original design space. In the rest of \S\ref{sec-model} we omit the very technical formulations of the statements and proofs, because these subsections are not used in the proofs of main results, however, they help in understanding of the whole concept.


\begin{definition} \label{def-contact} Let $f$ be a $C^n$ function in a disk $D\subset\mathbb{R}^2$. Let $(x(t),y(t),z(t))$, where $t$ runs through an interval $I$, be a $C^n$ curve such that $(\dot{x}(t),\dot{y}(t))\ne 0$ for each $t\in I$. We say that the curve \emph{has contact of order} $n$ with the graph of $f$ at $t=0$, if
\begin{equation*}
    \frac{z(t)-f(x(t),y(t))}{t^n}\to 0\text{ as }t\to 0.
  \end{equation*}
In particular, a curve intersecting the graph for $t=0$ has contact of order $0$; a curve tangent to the graph for $t=0$
has contact of order $1$, etc.; a curve fully contained in the graph has contact of \emph{infinite} order.
\end{definition}

Likewise, two  curves $c_1, c_2$ have contact of order $n$ at a common point if there are regular parameterizations $c_1(t), c_2(t)$
of these curves that agree for some $t=t_0$ in function value and derivatives up to the $n$-th order. A totally analogous definition
holds for two surfaces. Contact order $n$ between a curve $c$ and surface $\Phi$, as given in the above definition, can also be defined as follows: the surface $\Phi$ contains a regular smooth curve $c_1$ which has $n$-th order contact with the curve $c$. This curve $c_1 \subset \Phi$ is not uniquely determined. If there is one such
curve $c_1$, there are infinitely many other curves in $\Phi$ which verify $n$-th order contact with $c$. For example, consider a tangent line $c$ at an elliptic point and a pencil of planes that contain $c$. Then each plane of the pencil intersects $\Phi$ in a curve $c_1$ that each has the first order contact with $c$.

Consider two regularly parametrized curves $C_1^i(t), C_2^i(t)$ in the isotropic model which have contact of order $n \ge 1$ at some common point
$C_1^i(t_0)=C_2^i(t_0)$. The curves as point sets correspond to plane families in design space. Their envelopes are two developable surfaces $C_1, C_2$.
It is not hard to show that these developable surfaces have
a common ruling and contact of order $n$ at each point of the ruling (see \cite{pottwall:2001}). Let us now assume that we have contact of order $n$ between a curve $C^i$ and a surface $\Phi^i$ in the isotropic model.
In design space, this corresponds to a developable surface $C$ and a surface $\Phi$. However, $C$ and $\Phi$ do \emph{not}
have contact of order $n$ if we view these surfaces as point sets: For instance, if a cone $C$ is tangent to a sphere $\Phi$ along a circle, then $C^i$ is contained in $\Phi^i$, hence has contact of arbitrarily high order. But the rulings of $C$ have contact order just $1$ with the sphere $\Phi$.
We have to view $C$ and $\Phi$ as plane sets. This means that there exist (in fact, infinitely many) tangent developable surfaces of $\Phi$ which have $n$-th order contact with $C$. We illustrate this in the following at hand of examples that are very relevant for our setting.

\subsection*{Second order contact}

Let $C^i$ be a curve in the isotropic model. At each point $C^i(t_0)$, the curve has an {\it osculating isotropic circle} $C_o^i$. 
It has 2nd order contact with $C^i$ at $C^i(t_0)$, lies in the osculating plane and its top view is the Euclidean osculating circle of the top view of $C^i$, see Fig.~\ref{fig:Osculation}. In the original space, $C^i$ corresponds to a set of planes which envelope a certain developable surface $C$. The osculating isotropic circle corresponds to a cone of revolution $C_o$. It has 2nd order contact with the developable
surface $C$ along an entire common ruling and is called its {\it osculating cone} along that ruling. The vertex of the cone lies on
the (singular) regression curve of $C$ (see e.g. \cite{pottwall:2001}, Theorem 6.1.4).

Assume now that the curve $C^i$ lies on some surface $\Phi^i$. An isotropic osculating circle $C_o^i$ of $C^i$ has 2nd order
contact with $\Phi^i$. Mapping back to design space, we obtain a developable surface $C$ which is tangent to a surface $\Phi$
along some curve. The isotropic circle $C_o^i$ corresponds to an osculating cone $C_o$ of $C$. That cone does not have 2nd order contact with the surface, if one views the cone as a point set. However, the cone has 2nd order contact in plane space. This means that there exist tangent developable surfaces of $\Phi$ which have 2nd order contact with $C_o$. Among those tangent developables we can take a special one, namely the cone $C_1$ (not necessarily of revolution) which shares the vertex $v$ with $C_o$. It is enveloped by all tangent planes of $\Phi$ that pass through the point $v$ (or, equivalently, are tangent to the sphere with center $v$ and radius zero).
(In general position, $v$ does not belong to the surface $\Phi$, and in particular $v^i$ (isotropic
sphere) is transversal to $\Phi^i$ at the contact point of $C^i$ and $C^i_o$.) The intersection of $C_1$ with a plane $P$ is the contour of $\Phi$ under a central projection from the point $v$ onto the plane $P$. For example, if we take an image plane orthogonal to the axis of
$C_o$, its intersection with $C_o$ is a circle. This circle is the osculating circle of the contour of $\Phi$ for
projection from $v$ onto $P$.

 \begin{figure}[!tb]
\vrule width0pt
\begin{overpic}[width=0.45\columnwidth]{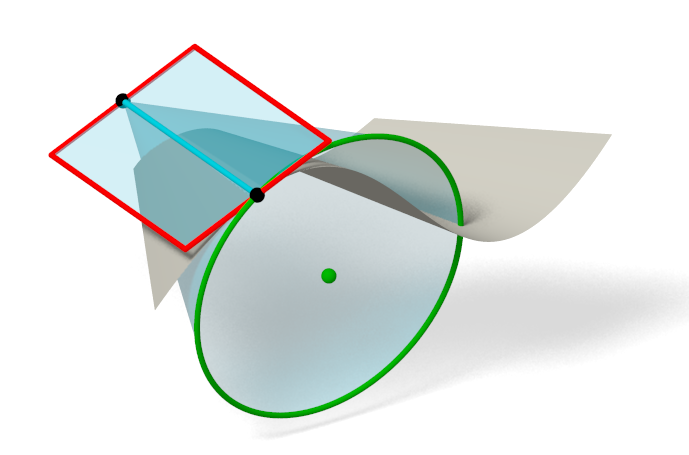}
{\small
	\put(0,-2){(a)}
	\put(60,8){$C_o$}
	\put(80,40){$C$}
  }
\end{overpic}
\hfill
\begin{overpic}[width=0.45\columnwidth]{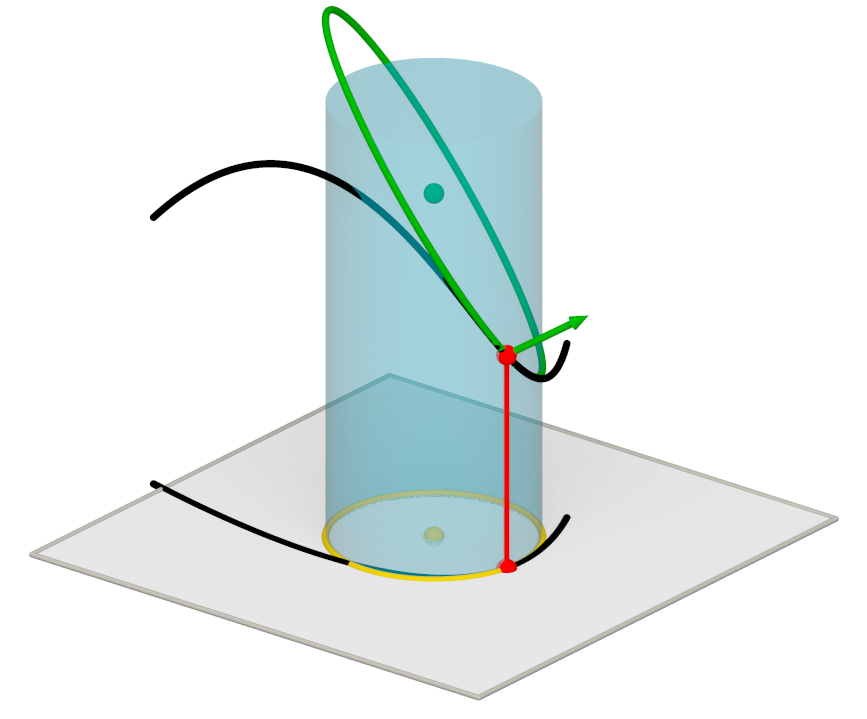}
{\small
	\put(0,-2){(b)}
	\put(75,20){$z=0$}
	\put(20,63){$C^i$}
	\put(48,40){$C^i(t_0)$}
	\put(32,76){$C_o^i$}	
  }
\end{overpic}
\hfill\vrule width0pt\\
\vspace{-12pt}
\caption{Osculating cone of a developable surface. a) In design space, a developable surface $C$ (grey) has second order contact with a cone $C_o$ (transparent, green base). They share a tangent plane along the common ruling (blue), however, they locally intersect.  b) In the isotropic model, the developable surface $C$ corresponds to a 3D curve $C^i$ and the osculating cone $C_o$ is mapped to an osculating isotropic circle $C_o^i$ (green) that has second order contact with $C^i$. The top view is a circle (yellow) in the plane $z=0$ that osculates the projection of $C^i$.}
\label{fig:Osculation}
\end{figure}


 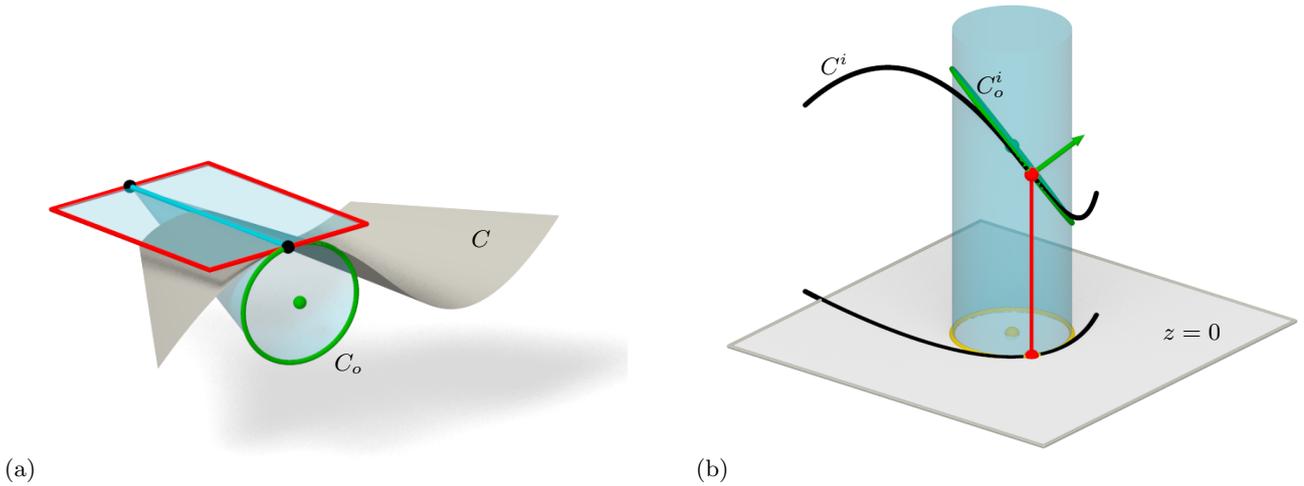
\begin{figure}[!tb]
\vrule width0pt
\begin{overpic}[width=0.45\columnwidth]{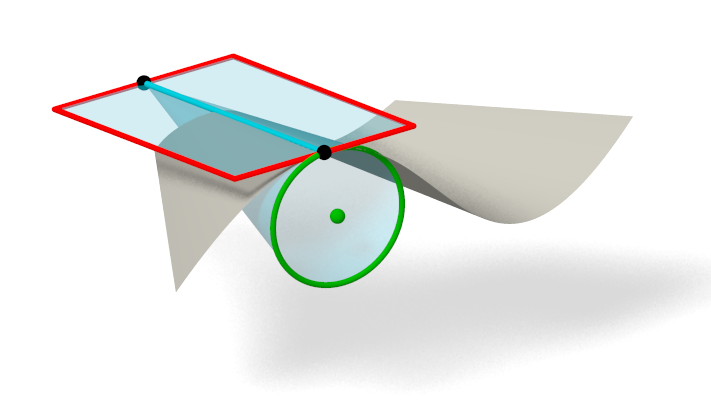}
{\small
	\put(0,-2){(a)}
	\put(53,15){$C_o$}
	\put(75,35){$C$}
  }
\end{overpic}
\hfill
\begin{overpic}[width=0.45\columnwidth]{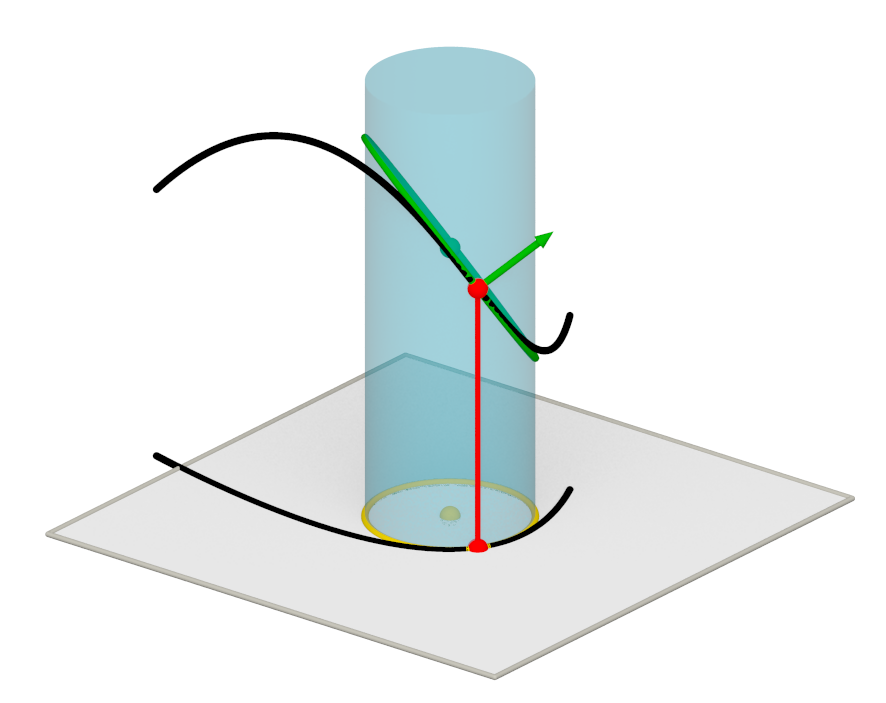}
{\small
	\put(0,-2){(b)}
	\put(75,20){$z=0$}
	\put(20,63){$C^i$}
	\put(45,60){$C_o^i$}	
  }
\end{overpic}
\hfill\vrule width0pt\\
\vspace{-12pt}
\caption{Hyperosculation. a) For certain positions of the ruling (blue) the osculating cone $C_o$ can have third order contact with a
developable surface $C$, i.e., it becomes hyperosculating.  b) In the isotropic model, this situation corresponds to hyperosculation between the isotropic circle (green) and the 3D curve $C^i$ that represents the developable surface $C$. In the top view, the isotropic circle (yellow) hyperosculates the projection of $C^i$.}
\label{fig:Hyperosculation}
\end{figure}

{\em Mannheim sphere.} Like in Euclidean geometry, there is a Meusnier's theorem
in isotropic geometry: Given a surface $\Phi^i$, and a point $P^i \in \Phi^i$
with a surface tangent $T^i$. Then, the osculating isotropic  circles of all curves $C^i \subset \Phi^i$ which
pass through $P^i$ with tangent $T^i$, lie in an isotropic sphere. Mapping back to design space, we
obtain Mannheim's theorem \cite{kruppa:1957}: The osculating cones of all developable surfaces $C$
which are tangent to a given surface $\Phi$ and have a common ruling $R$ (tangent to $\Phi$),  are tangentially circumscribed to a sphere ({\emph{Mannheim sphere}}). This gives an overview of all cones
which have 2nd order contact (as plane sets) with a given surface at a fixed point and allows one to apply additional constraints, for example, on the opening angle of the cone.

\subsection*{Third order contact}

In Euclidean geometry, there are results on circles which have third order contact with a given
surface at a given point. These have even been proposed for CNC machining with a cylindrical cutter
since the bottom circle of the cutter will actually generate the shape and thus 3rd order contact
leads to a good surface finish, at least in theory \cite{CurvCatering1-1993,CurvCatering2-1993,MillingCircles-2015}.
In fact, 2nd order contact is in general not enough because an osculating circle of a surface will locally change
the side of the surface and thus cause gauging. This local interference is not present for 3rd order
contact. In practice, it is hard to find a path which
leads the cutter in such a way that the bottom circle stays in third order contact with the surface.
However, knowing that third order contact is a limit of a double contact,
Kim et al.  \cite{MillingCircles-2015} used hyperosculating circles for initializing an optimization algorithm
which leads the cutter such that it has a double contact with the target surface.

We will derive analogous characterizations of hyperosculating circles in isotropic geometry,
which we expect to have applications in the original design space. There, one obtains hyperosculating
cones (in the plane geometric sense; see Fig.~\ref{fig:Hyperosculation}). Moreover, near such positions one can find
doubly tangent cones. This is not exploited for CNC machining in the present paper, but could be the topic of
future research.


 \begin{figure}[!tb]
\vrule width0pt
\begin{overpic}[width=0.45\columnwidth]{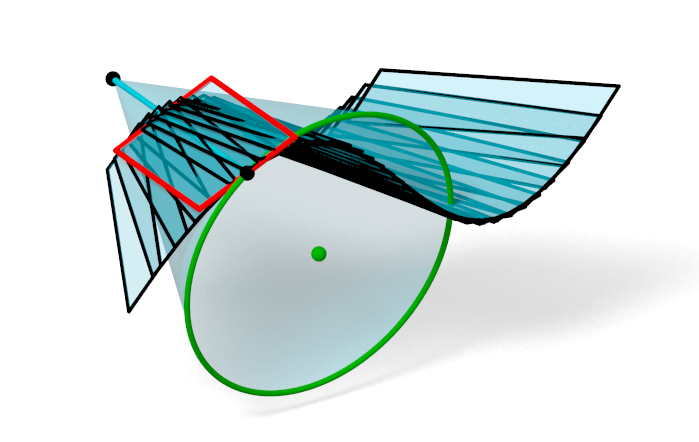}
{\small
	\put(0,-2){(a)}
  }
\end{overpic}
\hfill
\begin{overpic}[width=0.45\columnwidth]{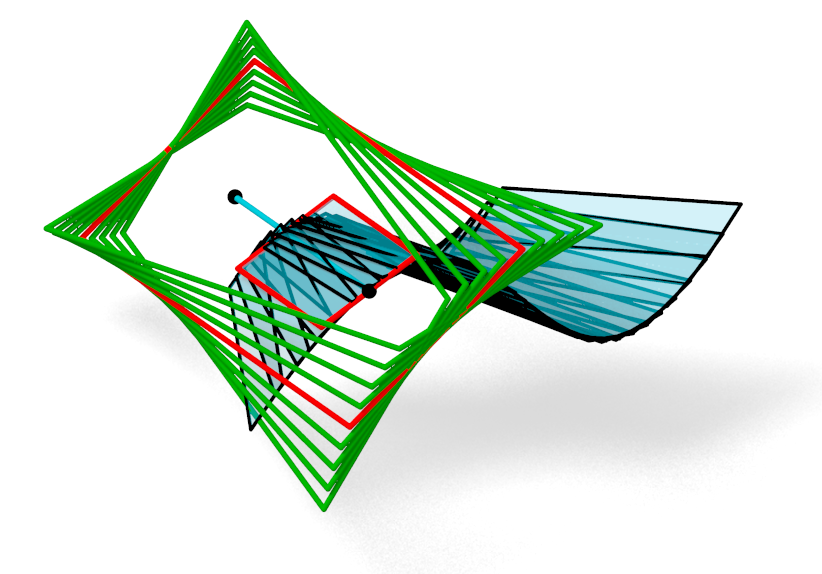}
{\small
	\put(0,-2){(b)}
  }
\end{overpic}
\hfill\vrule width0pt\\
\caption{Osculation in the space of planes. (a) A developable surface is represented as a one parameter family of tangent planes (transparent) and is osculated by a cone (green base) along a common tangent plane (red). (b) The osculating cone is also represented by one-parameter family of tangent planes (green) and the two families osculate at the red plane. Observe that the two red planes are identical.}
\label{fig:OsculationPlanes}
\end{figure}

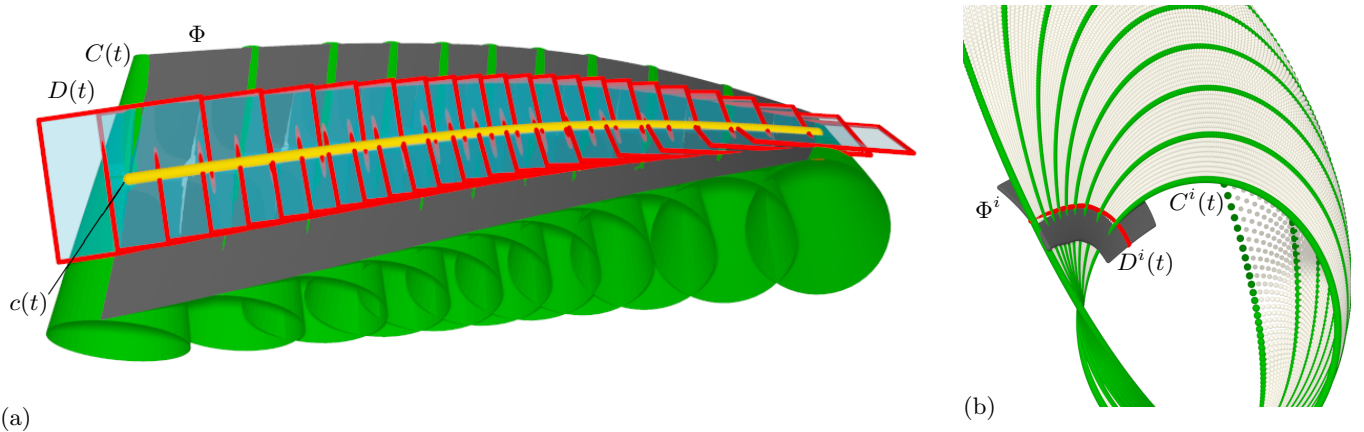
\begin{figure}[!tb]
\vrule width0pt\hfill
\begin{overpic}[width=0.68\columnwidth]{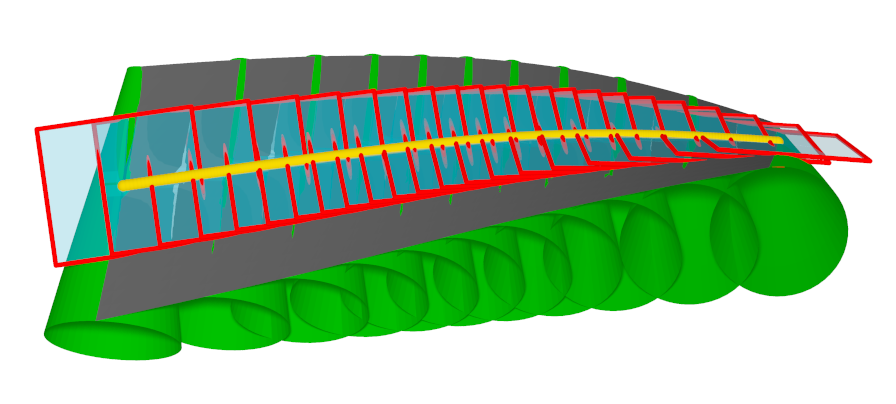}
{\small
    \put(0,-2){(a)}
	\put(20,39){$\Phi$}
    \put(9,37){$C(t)$}
    \put(5,33){$D(t)$}
   \put(5.2,12){\line(2,3){8}}
   \put(1,10.3){$c(t)$}
}
\end{overpic}
\hfill
\begin{overpic}[width=0.28\columnwidth]{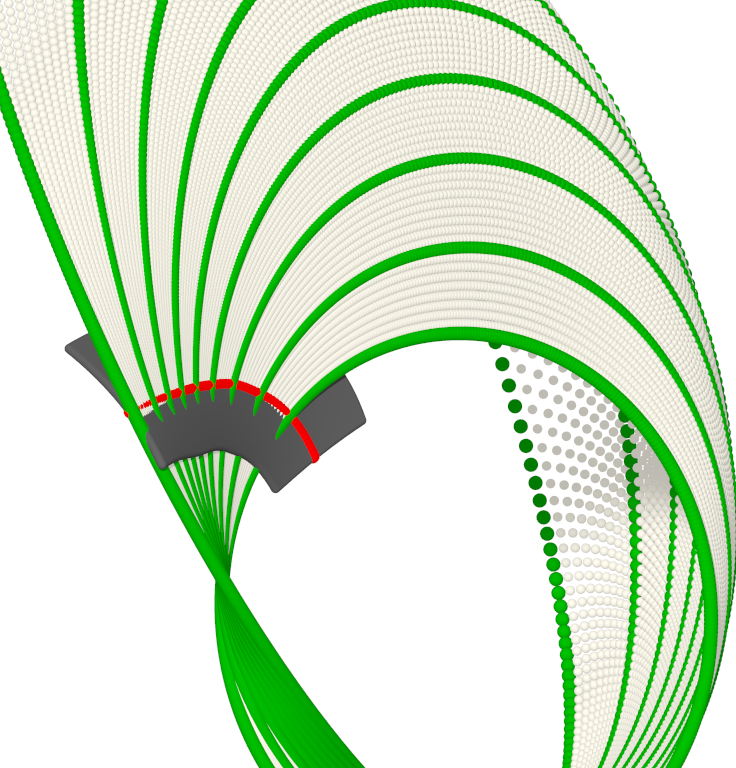}
{\small
  \put(0,-2){(b)}
   \put(51,49){$C^i(t)$}
   \put(38,34){$D^i(t)$}
   \put(3,47){$\Phi^i$}
  }
\end{overpic}
\hfill\vrule width0pt\\
\vspace{-12pt}
\caption{ Approximation of a general surface. (a) A reference surface $\Phi$ (grey) is approximated by an envelope of a moving cone (green) in the neighborhood of the contact curve $c(t)$ (yellow). Along this curve, they share a one-parameter family of tangent planes $D(t)$ (red). (b) In the isotropic space, the tangent planes are mapped to a curve $D^i(t)$ (red) that lies on the isotropic image  $\Phi^i$ of the surface. The tangent planes of each cone are mapped to an isotropic circle (green) 
intersecting the curve  $D^i$ transversely.}
\label{fig:PlaneContact}
\end{figure}

\subsection*{Higher order contact between a surface and an envelope of cones}

 Let us now return to the application in CNC machining with a conical tool. We want to approximate a given surface $\Phi$ with nowhere vanishing Gaussian curvature
(design surface) by an envelope of a moving rotational cone (tool). We claim that \emph{having contact of order $n$ between tool positions
and $\Phi$ in the space of planes guarantees
contact of order $n$ between $\Phi$ and
the envelope surface $\Psi$ generated by the moving
tool, and vice versa. We only have to make
sure that the tool is always moved into an appropriate} direction. We will now show that almost all directions are appropriate, except
for the one which is conjugate to the ruling of the tool at the contact point.

To explain that, let us first have a look into the isotropic model.
There, the tool positions $C(t)$ appear as isotropic circles $C^i(t)$ which have contact of order $n$ with $\Phi^i$. As long as the curve of
contact points $D^i(t)$ of these circles is \emph{transversal} to the circles, we have contact of order $n$ between
the surface $\Psi^i$ generated by the isotropic circles  (envelope in the isotropic model) and $\Phi^i$, see Fig.~\ref{fig:PlaneContact}. Mapping back to $\mathbb{R}^3$,
we obtain contact of order $n$ between the surface $\Phi$ and the envelope $\Psi$. We just have to clarify how to recognize
the mentioned transversality in the isotropic model directly in design space $\mathbb{R}^3$. The contact curve $D^i$ corresponds to a developable surface $D \subset \mathbb{R}^3$ which is the envelope of the common tangent planes $D(t)$ of the moving tool and $\Phi$ at the cutter contact points. In other words, this developable surface $D$ is tangent to $\Phi$ along the set of cutter contact points $c(t)$,
which form a curve $c \subset \Phi$.  The mentioned transversality in the isotropic model means that none
of the isotropic circles $C^i(t)$ is tangent to the contact curve $D^i$. As we have already discussed, two curves in the isotropic
model share a common point and tangent,
if the corresponding developable surfaces
in Euclidean space share a common tangent plane and ruling.
Hence, the cone rulings $r_C(t)$ through the contact points $c(t)$ have to be different from the rulings $r_D(t)$ of $D$.

Essentially, we are still in the space of planes. To get information about an appropriate contact curve $c(t)$, we recall
a classical result: Given a curve $c$ on a surface $\Phi$, the developable surface $D$ which is tangent to $\Phi$
along $c$ has rulings $r_D(t)$ which are conjugate to the tangents of $c$ (see, e.g., \cite{pottwall:2001}, page 334).
Rulings $r_C(t)$ and $r_D(t)$ are different, if their conjugate directions with respect to $\Phi$ are different.
Hence, the tangent of the cutter contact curve $c$ has to be different from the direction
which is conjugate to the cutter's ruling $r_C(t)$. Ideally, one will want to move
the cutter orthogonal to that conjugate direction to obtain the widest machined strips.

Note that we are interested here in contact of order $n\ge 2$, which essentially means $n=2$
or $n=3$. There, the envelope of the moving tool and the target surface $\Phi$ share the conjugacy
relation at the contact points. The conjugate direction of the cone ruling $r_C(t)$ at the contact point
with respect to the envelope of the cones is the tangent to the characteristic (since the cone is the tangent
developable of the envelope along the characteristic). Hence, an appropriate direction of the tool movement is one which is transversal to the characteristic. This is exactly what one would expect. But note that when we want to plan the motion and want to move from one position to the
next, we can use the conjugate direction of $r_C(t)$ with respect to $\Phi$. Moving as orthogonal
as possible to that conjugate direction and satisfying other machining constraints, will lead to a next appropriate cone position. For actual machining, osculation ($n=2$) will lead to undercutting. However, one can use a motion with an osculating envelope as a guide and work with a slightly smaller cutting tool to avoid undercutting. Recall once again that the discussion throughout this section is valid under certain technical assumptions like general position and non-vanishing Gaussian curvature.

\section{Surfaces containing a special conic through each point}
\label{sec-conics}

In this section we characterize the surfaces containing a conic satisfying condition~$(\Theta)$ through each point (Theorem~\ref{conj-reciprocal} below). This is similar to the characterization of ruled surfaces in Section~\ref{sec-developable}. We consider the graph of a smooth function $f$.
The conics on the graph are parametrized by trigonometric functions. Differentiation with respect to the parameter gives a system of algebraic equations on the tangential direction to the top view of the conic at a given point. Solvability of the system is the required condition on $f$. Conversely, given $f$ such that the system has a solution, the conics are reconstructed as follows. First we pick up a suitable solution to get a vector field tangential to the top views of the future conics. Then we prove that the integral curves of the field are circles. Finally we show that the restriction of $f$ to these circles is linear. Such reconstruction works under the minor restrictions that the solution of the system is not multiple and continuously depends on the point.

\subsection*{Conics parametrization}

\begin{proposition}\label{prop-conic} Each conic satisfying condition~$(\Theta)$ can be parametrized as
\begin{equation}\label{eq-conj-reciprocal}
  \begin{cases}
    x(t) &= x+v\sin t+u(1-\cos t), \\
    y(t) &= y-u\sin t+v(1-\cos t), \\
    z(t) &= z+a\sin t+b(1-\cos t),
  \end{cases}
  \end{equation}
where $a,b,u,v,x,y,z\in \mathbb{R}$ satisfy
\begin{equation}\label{eq-conj-characterization1}
\left(x^2+y^2+1+2xu+2yv\right)^2
-4\tan^2\theta\left(u^2+v^2\right)=0.
\end{equation}
\end{proposition}

\begin{proof}[Proof of Proposition~\ref{prop-conic}]
Let $(x,y,z)$ be a point on the conic and $(x+u,y+v)$ be the center of the top-view circle. Clearly, then the circle is parametrized by $x(t)$ and $y(t)$ from~\eqref{eq-conj-reciprocal}. Since a conic is a planar curve, $z(t)$ must be a linear function in $x(t)$ and $y(t)$, and we arrive at~\eqref{eq-conj-reciprocal} for some $a,b\in\mathbb{R}$.

Now turn to condition~$(\Theta)$. Let
$A$ and $B$ be the points of the top-view circle which are the closest and the furthest from the origin $O$ (or just opposite points, if $O$ is the center). The inverse stereographic projection of the circle from the point $S=(0,0,-1)$ is a circle of intrinsic radius $\angle ASB$ in the unit sphere. Thus
\begin{multline*}
\pm\cot \theta
=\tan ASB
=\frac{\tan BSO\mp\tan ASO}{1\pm \tan ASO\tan BSO}
=\frac{OB\mp OA}{1\pm OA\cdot OB}
=\frac{2\sqrt{u^2+v^2}}{1+(x+u)^2+(y+v)^2-u^2-v^2},
\end{multline*}
where the choice of sign in the left-hand side depends on if
$\angle ASB$ is acute or obtuse, and the other signs
depend on if $O$ is outside or inside the top-view circle.
We arrive at~\eqref{eq-conj-characterization1}.
\end{proof}

\subsection*{Derivation of the system}

Let us derive PDEs for functions whose graphs contain a conic satisfying condition~($\Theta$) through each point.
We are not actually using that the entire conic is contained in the graph; a sufficiently high contact suffices.

\begin{example} The graph of the function $f(x,y)=\frac{y^2}{x^2+y^2}$, where $(x,y)\ne (0,0)$, is covered by a $2$-dimensional family of conics~\eqref{eq-conj-reciprocal} with $u=-\frac{x}{2}$, $v=-\frac{y}{2}$, $z=\frac{y^2}{x^2+y^2}$, $a=\frac{xy}{x^2+y^2}$, $b=\frac{x^2-y^2}{2(x^2+y^2)}$. The ones with $x^2+y^2=\cot^2\theta$ satisfy condition~$(\Theta)$. We return to this example in Section~\ref{sec-results}.
\end{example}

\begin{proposition}\label{prop-tangency-order} Conic~\eqref{eq-conj-reciprocal} has contact of order $2$ with the graph of $f$ (``\emph{osculation}''), if and only if
\begin{equation}\label{eq-prop-tangency-order2}
\begin{cases}
z=f(x,y),&\\
a=f_x v-f_y u, &\\
b=f_x u+f_y v+f_{xx}v^2-2f_{xy}uv+f_{yy}u^2. &
\end{cases}
\end{equation}
The contact order is $3$  (``\emph{hyperosculation}''), if and only if in addition
\begin{equation}\label{eq-conj-characterization2}
f_{xxx}v^3-3f_{xxy}v^2u+3f_{xyy}vu^2-f_{yyy}u^3 
+3(f_{xx}-f_{yy})uv+3f_{xy}(v^2-u^2)=0.
\end{equation}
The contact order is $4$, if and only if in addition
\begin{multline}\label{eq-conj-characterization3}
f_{xxxx}v^4-4f_{xxxy}v^3u+6f_{xxyy}v^2u^2-4f_{xyyy}vu^3+f_{yyyy}u^4 \\
\qquad +6 u v^2 f_{xxx}+
6v(v^2-2u^2)f_{xxy}+6u(u^2-2v^2)f_{xyy}+6 u^2 v f_{yyy} \\
\qquad +3 (u^2- v^2)(f_{xx}-f_{yy})+12uvf_{xy}=0.
\end{multline}
\end{proposition}

\begin{proof}[Proof of Proposition~\ref{prop-tangency-order}]
   The proof is by consecutive differentiation of $z(t)-f(x(t),y(t))$ with respect to $t$ and evaluating at $t=0$. For instance, the second derivative is
   \begin{multline*}
   -a\sin t+b\cos t-f_x(u\cos t-v\sin t)-f_y(v\cos t+u\sin t)\\
   -f_{xx}(v\cos t+u\sin t)^2-2f_{xy}(v\cos t+u\sin t)(v\sin t-u\cos t)
   -f_{yy}(v\sin t-u\cos t)^2.
   \end{multline*}
   For contact of order $2$, this must vanish at $t=0$, which gives the third equation of~\eqref{eq-prop-tangency-order2}.
\end{proof}

Equation~\eqref{eq-prop-tangency-order2} is an expression that links together the point $(x,y,z)$ in the isotropic space, the center $(x+u,y+v,0)$ of the Euclidean circle in the plane $z=0$, two parameters $a$ and $b$ that control the inclination of the plane that contains the isotropic circle, and the derivatives of the function $f$; see Fig.~\ref{fig:IsotropicCircle}. The remaining equations~\eqref{eq-conj-characterization2}--\eqref{eq-conj-characterization3} together with~\eqref{eq-conj-characterization1} give a nontrivial restriction on the function $f$ itself.

 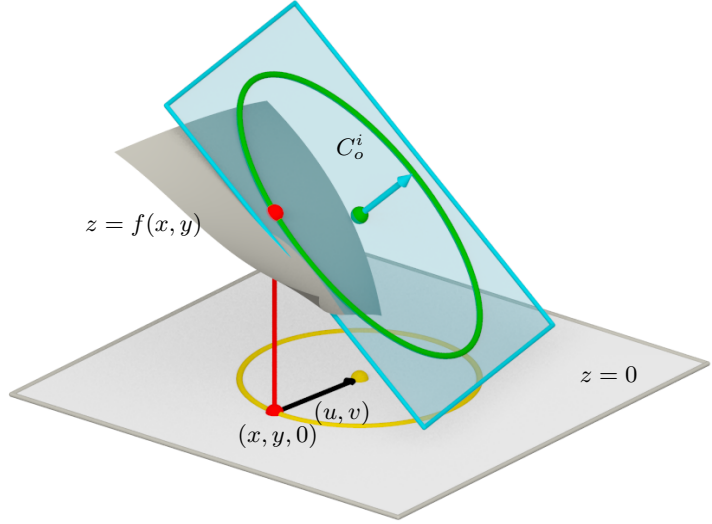
\begin{figure}[!tb]
\vrule width0pt \hfill
\begin{overpic}[width=0.55\columnwidth]{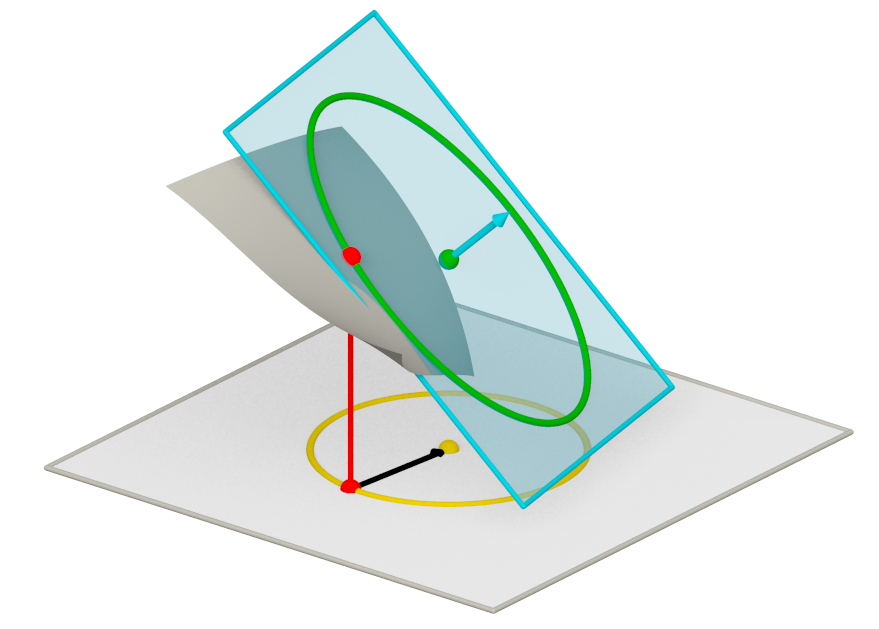}
{\small
	\put(15,40){$z=f(x,y)$}
	\put(35,12){$(x,y,0)$}
	\put(45,15){$(u,v)$}
	\put(80,20){$z=0$}
	\put(48,50){$C_o^i$}
  }
\end{overpic}
\hfill\vrule width0pt\\
\vspace{-12pt}
\caption{Osculating isotropic circle in a given direction. The graph of a function $f(x,y)$ (gray) is osculated by an isotropic circle $C_o^i$ (green) at a point $(x,y,f(x,y))$. The top view of the osculating isotropic circle is a Euclidean circle (yellow) that passes through $(x,y,0)$ and is centered at $(x+u,y+v,0)$. The constraints for osculation between $f$ and $C_o^i$ are given in Eq.~\eqref{eq-prop-tangency-order2}.}
\label{fig:IsotropicCircle}
\end{figure}

\begin{corollary}\label{conj-characterization}
Let $f$ be a $C^4$ function in a disk $D\subset\mathbb{R}^2$.
If through each point of the surface $z=f(x,y)$ there passes an arc of a conic satisfying condition~$(\Theta)$ and completely contained in the surface, then for each $(x,y)\in D$ three equations~\eqref{eq-conj-characterization1},\eqref{eq-conj-characterization2},\eqref{eq-conj-characterization3} have a common real solution $(u,v)$.
\end{corollary}


Let us discuss this result from a computational viewpoint: The existence of a common real solution implies (but is not equivalent to) vanishing of the resultant of the $3$ polynomials in $(u,v)$ in the left-hand sides of the equations. However, one ends up with a huge expression that is hardly useful. In contrast, it is computationally more economical to first solve the system of first \emph{two} equations~\eqref{eq-conj-characterization1} and~\eqref{eq-conj-characterization2}, and then verify if the third one~\eqref{eq-conj-characterization3} is satisfied. 

To solve the system of \emph{two} equations~\eqref{eq-conj-characterization1} and~\eqref{eq-conj-characterization2} (under additional general position assumptions), one first finds the real roots $t$ of the polynomial
\begin{multline}\label{eq-simplified-system}
    (x^2+y^2+1)\left(\left(t^2-1\right)^3f_{xxx}-6 \left(t^2-1\right)^2 t f_{xxy}+12 \left(t^2-1\right) t^2 f_{xyy}-8 t^3 f_{yyy}\right)\\
    +
    6 \left(2 \left(t^3-t\right) (f_{xx}-f_{yy})+\left(t^4-6 t^2+1\right) f_{xy}\right) \left(t^2\tan\theta-t^2y-2 t x+y+\tan\theta\right)=0,
\end{multline}
which has degree $6$  unless $(x^2+y^2+1)f_{xxx}+6f_{xy}(\tan\theta-y)=0$,
and then comes up with
\begin{equation*}
u=\frac{t(x^2+y^2+1)}{t^2\tan\theta-t^2y-2 t x+y+\tan\theta},\qquad
v=\frac{(t^2-1)(x^2+y^2+1)}{2(t^2\tan\theta-t^2y-2 t x+y+\tan\theta)},
\end{equation*}
unless the denominators vanish. For the numerical stability of this approach, 
\eqref{eq-simplified-system} should not have multiple roots.



\subsection*{Reconstruction of conics}

We are able to prove the reciprocal assertion of Corollary~\ref{conj-characterization} under the minor restrictions that the common solution $(u,v)$ is \emph{not a multiple} root of  \eqref{eq-conj-characterization1} and \eqref{eq-conj-characterization2}, and continuously depends on the point $(x,y)$. We say that conic~\eqref{eq-conj-reciprocal} is \emph{multiple}, if $(u,v)$ is a common real multiple root of \eqref{eq-conj-characterization1} and \eqref{eq-conj-characterization2}.



\begin{theorem}\label{conj-reciprocal}
Let $f$ be a $C^4$ function in a disk $D\subset\mathbb{R}^2$.
Suppose that through each point $(x,y,z)$ of the graph of $f$,
there passes an arc of a nonmultiple conic $C_{x,y}$ having contact order~$4$  at $(x,y,z)$ with the graph, continuously depending on $(x,y)$, and such that the top view of $C_{x,y}$ is the stereographic projection of a circular arc of intrinsic radius $\frac{\pi}{2}-\theta$ (not passing through the projection center). Then an arc of the conic $C_{x,y}$ is contained in the graph.
\end{theorem}

\begin{remark} By Proposition~\ref{prop-tangency-order},
the assumptions of Theorem~\ref{conj-reciprocal} are equivalent to equations~\eqref{eq-conj-characterization1},\eqref{eq-conj-characterization2},\eqref{eq-conj-characterization3} having a common real solution $(u,v)$ \emph{nowhere} satisfying the equation (where the left-hand side is the Jacobian of
\eqref{eq-conj-characterization1} and \eqref{eq-conj-characterization2})
\begin{multline}\label{eq-conj-reciprocal-obstruction}
  f_{xxx}v^2\tilde u+f_{xxy}v(v\tilde v-2u\tilde u)+f_{xyy}u(u\tilde u-2v\tilde v)+f_{yyy}u^2\tilde v
  +(f_{xx}-f_{yy})(u\tilde u-v\tilde v)+2f_{xy}(u\tilde v+v\tilde u)=0,
\end{multline}
where
\begin{equation}\label{eq-conj-reciprocal-tildes}
\tilde u=x(x^2+y^2+1+2xu+2yv)- 4u\tan^2\theta,\qquad
\tilde v=y(x^2+y^2+1+2xu+2yv)- 4v\tan^2\theta.
\end{equation}
\end{remark}


The restriction that the conic continuously depends on the point seems inessential; it is imposed to bypass technical issues discussed in Sections~\ref{sec-developable}--\ref{sec-ruled}. But dropping the restriction that the conic is nonmultiple would require new ideas, just like developable surfaces require special treatment in characterization of ruled surfaces in~Section~\ref{sec-ruled}.

\begin{problem} Prove the reciprocal assertion in Corollary~\ref{conj-characterization} in the case when \eqref{eq-conj-reciprocal-obstruction} holds identically, i.e. the conic is multiple.
Is it true that in this case the surface $z=f(x,y)$ is the envelope of a one-parametric family of rotational paraboloids~\eqref{iso-sphere} such that each characteristic is 
a conic satisfying condition~$(\Theta)$?
\end{problem}

\begin{proof}[Proof of Theorem~\ref{conj-reciprocal}]

First let us show that the conic $C_{x,y}$ smoothly depends on $x$ and $y$, more precisely, that it is parametrized by~\eqref{eq-conj-reciprocal} with the coefficients being $C^1$ functions in $x$ and $y$. Indeed, by Proposition~\ref{prop-tangency-order} the conic $C_{x,y}$ is given by~\eqref{eq-conj-reciprocal} for some $(u,v,a,b,z)=(u(x,y),v(x,y),a(x,y),b(x,y),z(x,y))$
continuously depending on $(x,y)$ and satisfying four equations~\eqref{eq-conj-characterization1}--\eqref{eq-conj-characterization3}.
Since $C_{x,y}$ is nonmultiple, it follows that $(u,v)$ is not a multiple solution of the system of equations~\eqref{eq-conj-characterization1} and~\eqref{eq-conj-characterization2}. Hence by the implicit function theorem, it follows that $u(x,y)$ and $v(x,y)$ are $C^1$. By~\eqref{eq-prop-tangency-order2} the remaining coefficients are $C^1$ as well.

Notice that the vector $(v,-u)$ is tangent to the top view of conic~\eqref{eq-conj-reciprocal} at the point $t=0$. Integrate the resulting vector field:
by the Picard-Lindel\"of theorem for some $\epsilon>0$ there is a regular curve $(x(t),y(t))$ such that $\dot{x}(t)=v(x(t),y(t))$ and $\dot{y}(t)=-u(x(t),y(t))$ for each $t\in (-\epsilon,\epsilon)$. 

Let us prove that $(x(t),y(t))$ is a circular arc, and  moreover $(x(t),y(t),f(x(t),y(t)))$ is contained in the conic $C_{x(0),y(0)}$. Hereafter all equations are understood as holding for each $t\in (-\epsilon',\epsilon')$, where possibly $\epsilon'<\epsilon$, and all the functions $u=u(x(t),y(t)), v=v(x(t),y(t)), f=f(x(t),y(t))$ and their derivatives are evaluated at the point $(x(t),y(t))$.

Start with equation~\eqref{eq-conj-characterization1}.
Differentiating it with respect to $t$ and substituting $\dot{x}=v$, $\dot{y}=-u$, we get
  $$
  2(x(v+\dot{u})+y(\dot{v}-u))(x^2+y^2+1+2xu+2yv)- 8\tan^2\theta (u\dot{u}+v\dot{v})=0.
  $$
  This is equivalent to (recall notation~\eqref{eq-conj-reciprocal-tildes})
  $$
  (v+\dot{u})\underbrace{\left( x(x^2+y^2+1+2xu+2yv)- 4u\tan^2\theta \right)}_{\tilde u}=
  (u-\dot{v})\underbrace{\left( y(x^2+y^2+1+2xu+2yv)- 4v\tan^2\theta\right)}_{\tilde v}.
  $$
  Here $(\tilde u,\tilde v)$ is half of the gradient of the left-hand side of~\eqref{eq-conj-characterization1} considered as a function in $(u,v)$. Thus $(\tilde u,\tilde v)\ne(0,0)$ because otherwise $(u,v)$ would be a multiple solution of the system of equations~\eqref{eq-conj-characterization1} and~\eqref{eq-conj-characterization2}. This implies that there is a function $g=g(t)$ (e.g., $g=(\dot{v}-u)/\tilde{u}$ for $\tilde{u}\ne 0$) such that
  \begin{equation}\label{eq-tilde}
  \begin{cases}
    \dot{u} &= -v - g\tilde{v}, \\
    \dot{v} &=  u + g\tilde{u}.
  \end{cases}
  \end{equation}
  Here we have essentially relied on a particular form of the constraint~~\eqref{eq-conj-characterization1}.

  Let us now switch to~\eqref{eq-conj-characterization2}. Differentiating with respect to $t$, substituting~\eqref{eq-tilde} and subtracting~\eqref{eq-conj-characterization3}, we get
  \begin{multline*}
  f_{xxxx}v^4-4f_{xxxy}v^3u+6f_{xxyy}v^2u^2-4f_{xyyy}vu^3+f_{yyyy}u^4 
 +3 u v^2 f_{xxx}+3v(v^2-2u^2)f_{xxy}+3u(u^2-2v^2)f_{xyy}+3 u^2 v f_{yyy}\\
 +3 v^2 \dot{v} f_{xxx} -3v(2u\dot{v}+v\dot{u})f_{xxy}+3u(2v\dot{u}+u\dot{v})f_{xyy}-3 u^2 \dot{u} f_{yyy} 
 +3 (u\dot{v}+ v\dot{u})(f_{xx}-f_{yy})+6(v\dot{v}-u\dot{u})f_{xy} \\
= 3g\left(
f_{xxx}v^2\tilde u+f_{xxy}v(v\tilde v-2u\tilde u)+f_{xyy}u(u\tilde u-2v\tilde v)+f_{yyy}u^2\tilde v\right.
  +\left.(f_{xx}-f_{yy})(u\tilde u-v\tilde v)+2f_{xy}(u\tilde v+v\tilde u)
\right) 
=g J=0,
\end{multline*}
  where $J$ is the Jacobian of the system of equations~\eqref{eq-conj-characterization1} and~\eqref{eq-conj-characterization2} in $u$ and $v$.
  Here $J\ne 0$ because $(u,v)$ is not a multiple solution; thus $g(t)=0$ identically.

  (The expression $g J$ in the right-hand side is what one should actually expect: The left-hand side is obviously linear in $g$ and vanishes for $g=0$ because \eqref{eq-conj-characterization3} was obtained from~\eqref{eq-conj-characterization2} by differentiating  along a circle. The coefficient before $g$ is $J$ because $(\tilde u,\tilde v)$ is half of the gradient of the left-hand side of~\eqref{eq-conj-characterization1}.)

  Then by~\eqref{eq-tilde} we get $(\dot{u},\dot{v}) = (-v,u)$. Here $(-v,u)\ne 0$ by~\eqref{eq-conj-characterization1}. Thus the integral curve $(x(t),y(t))$ is a circular arc parametrized by the first two equations of~\eqref{eq-conj-reciprocal} with $x=x(0), y=y(0), u=u(x(0),y(0)), v=v(x(0),y(0))$.

  In particular, this means that the integral curve is contained in the top view of the conic $C_{x(0),y(0)}$. By a similar argument, the same curve is contained in the top view of each conic $C_{x(t),y(t)}$ (which a priori can be different from $C_{x(0),y(0)}$).

  Clearly, for each fixed $a,b,z$ and each $t$ the function $z(t)$ given by the third equation of~\eqref{eq-conj-reciprocal} satisfies
  $$
  \frac{d^3 z(t)}{dt^3}+\frac{d z(t)}{dt}=0.
  $$
  Since $C_{x(t),y(t)}$ has contact of order $3$ with the graph of $f$, it follows that for each $t$ sufficiently close to $0$ also
  $$
  \frac{d^3 f(x(t),y(t))}{dt^3}+\frac{d f(x(t),y(t))}{dt}=0.
  $$
  Thus $f(x(t),y(t))=z'+a'\sin t+b'(1-\cos t)$ for some $z',a',b'\in\mathbb{R}$. Since the graph of $f$  has contact  of order $2$ with $C_{x(0),y(0)}$, it must contain an arc of the conic $C_{x(0),y(0)}$ completely.
\end{proof}


\section{Surfaces enveloped by a family of cones: conclusion} \label{sec-cones}

Now we use the results of the previous two sections to complete the characterization of surfaces enveloped by a one-parametric family of congruent cones (Theorem~\ref{cor-cone-envelope} below). 
We show how to construct $\Phi^i$ from $\Phi$ and vice versa. Then we show how to reconstruct the positions of cones in the enveloping family. In our implementation, when constructing $\Phi^i$ locally at some point $\mathbf{r}$, we rotate $\Phi$ such that the normal of $\Phi$ at $\mathbf{r}$ coincides with $(0,0,1)$. Globally, when a whole surface is considered, the mean normal vector is computed and aligned with $(0,0,1)$. This preprocessing aims at minimizing the distortion of the mapping
Eq.~\eqref{eq-isotropic-model}.

\subsection*{Characterization of surfaces enveloped by a family of cones}

We summarize now the previous results. From Proposition~\ref{prop-translation}, Corollary~\ref{conj-characterization}, and Theorem~\ref{conj-reciprocal} together we get the following characterization.

\begin{theorem}[characterization of surfaces enveloped by a family of cones] \label{cor-cone-envelope}
  Assume (*).

  If through each point of $\Phi$ there passes a cone which is tangent to $\Phi$ along a curve (containing the point), has the opening angle $\theta$, and has no tangent planes orthogonal to $(0,0,-1)$, then for each $(x,y)\in D$ three equations~\eqref{eq-conj-characterization1},
  \eqref{eq-conj-characterization2},\eqref{eq-conj-characterization3} have a common nonzero real solution $(u,v)$.

  Conversely, if for each $(x,y)\in D$ three equations~\eqref{eq-conj-characterization1},
  \eqref{eq-conj-characterization2}, \eqref{eq-conj-characterization3} have a common real solution $(u,v)$ continuously depending on $(x,y)$ and nowhere satisfying~\eqref{eq-conj-reciprocal-obstruction}, then through each point of $\Phi$ there passes a cone which is tangent to $\Phi$ along a continuous curve (containing the point) and has the opening angle $\theta$.
 \end{theorem}


We recall that among envelopes of cones, the ones with positive Gaussian curvature cannot be practically milled with a conical tool.

\begin{proposition}\label{prop-positive-curvature}
  If a surface is tangent to a cone along a curve and has positive Gaussian curvature, then it has common points with the interior of the cone (hence is not milable).
\end{proposition}

\begin{proof}
  Let $\mathbf{r}$ be a point on the surface $\Phi$. Since the Gaussian curvature is positive, it follows that $\Phi$ locally is contained in one half-space with respect to the tangent plane $T_r\Phi$. Let $C$ be the cone tangent to $\Phi$ along a curve passing through $\mathbf{r}$. Thus $C$ must be locally contained in the same half-space with respect to the tangent plane $T_r\Phi$. Consider the normal section to $\Phi$ at $\mathbf{r}$ passing through the ruling of $C$ through $\mathbf{r}$. The points of the section belong to the interior of $C$.
  \end{proof}

\subsection*{Construction of the surface in the isotropic model}

To apply Theorem~\ref{cor-cone-envelope} in practice, one needs to construct surface $\Phi^i$ from $\Phi$ and vice versa. These constructions are given by the following proposition.

\begin{proposition}\label{cor0}\label{cor1} \textup{ (Cf.~\cite[Corollary~2]{pottmann-2009-lms})}
Assume (*). Let $(n_1,n_2,n_3)$ be the oriented unit normal at a point $(r_1,r_2,r_3)$ of $\Phi$. Then the function $f$ and its derivatives are given by
\begin{align}
  f  \left(\frac{n_1}{n_3+1},\frac{n_2}{n_3+1}\right) &=
  -\frac{n_1 r_1+n_2 r_2+n_3 r_3}{n_3+1},\label{eq-f}\\
  f_x\left(\frac{n_1}{n_3+1},\frac{n_2}{n_3+1}\right) &=
  \frac{n_1 r_3}{n_3+1} -r_1,\label{eq-fx} \\
  f_y\left(\frac{n_1}{n_3+1},\frac{n_2}{n_3+1}\right) &=
  \frac{n_2 r_3}{n_3+1} -r_2.\label{eq-fy}
\end{align}
Conversely, given the function $f$, the surface $\Phi$ can be parametrized as follows:
\begin{equation}\label{eq-cor1}
\mathbf{r}(x,y)=
\frac{1}{x^2+y^2+1}\left(
\text{\begin{tabular}{c}
$(x{}^2-y{}^2-1)f_x+2x y f_y-2x f$\\
$(y{}^2-x{}^2-1)f_y+2x y f_x-2y f$\\
$2x f_x+2y f_y-2 f$
\end{tabular}}
\right).
\end{equation}
Here the point $\mathbf{r}(x,y)$ is the tangency point of $\Phi$ and the plane $P$ such that $P^i=(x,y,f(x,y))$.
\end{proposition}

In what follows we use the formula for the inverse stereographic projection from $(0,0,-1)$:
\begin{equation}\label{eq-inverse-stereographic}
\mathbf{n}(x,y)=
\frac{\left(2x, 2y, 1-x^2-y^2 \right)}{x^2+y^2+1}.
\end{equation}

\begin{proof}
The oriented tangent plane $P$ to $\Phi$ at the point $(r_1,r_2,r_3)$ is given by $n_1x+n_2y+n_3z-n_1r_1-n_2r_2-n_3r_3=0$. By the definition of $P^i$ we get~\eqref{eq-f}. Now let $(x,y)=\left(\tfrac{n_1}{n_3+1},\tfrac{n_2}{n_3+1}\right)$ be the stereographic projection of $(n_1,n_2,n_3)$. Then $\tfrac{n_3}{n_3+1}=\tfrac{1}{2}(1-x^2-y^2)$ by~\eqref{eq-inverse-stereographic}.
Substituting these expressions into~\eqref{eq-f}, differentiating with respect to $x$,
and using the condition $n_1\tfrac{\partial}{\partial x}r_1+n_2\tfrac{\partial}{\partial x}r_2+n_3\tfrac{\partial}{\partial x}r_3=0$ that $(n_1,n_2,n_3)$ is normal to~$\Phi$, we get~\eqref{eq-fx}.
Analogously we get~\eqref{eq-fy}. Solving ~\eqref{eq-f}--\eqref{eq-fy} as a linear system in $r_1,r_2,r_3$ we get~\eqref{eq-cor1}.
\end{proof}

\subsection*{Reconstruction of the cones}

To determine the position of a cone $C$ with a given opening angle and tangent to a given surface at a given point, it suffices to identify the position of the vertex and the side of the tangent plane which the cone borders upon at the tangency point (i.e. the halfspace containing a small neighborhood of the tangency point on the cone).

The vertex is reconstructed from the conic $C^i$ as follows.

\begin{proposition}\label{prop-vertex} Let $C$ be the cone such that the conic $C^i$ is parametrized by~\eqref{eq-conj-reciprocal}; then the vertex of $C$ is
\begin{multline}\label{eq-prop-vertex}
\mathbf{m}(x,y)=
\frac{1}{(u^2+v^2)(x^2+y^2+1+2 u x+2 v y)}\\
\times\left(
\text{\begin{tabular}{c}
$(x^2 - y^2 - 1) (a v + b u )  +
 2 x y (b v - a u)  -
 2 (u^2 + v^2) (u z + x z + a y)$\\
$( y^2 - x^2 - 1 ) (b v - a u) +
 2 x y (a v + b u )  -
 2 (u^2 + v^2) (v z + y z - a x)$\\
$2 x (a v + b u ) +
2 y (b v - a u) -
2 (u^2 + v^2) z$
\end{tabular}}
\right).
\end{multline}
\end{proposition}

\begin{proof}
  Let $(m_1,m_2,m_3)$ be the vertex of $C$. Then  conic~\eqref{eq-conj-reciprocal} must be contained in the surface~\eqref{iso-sphere} with $R=0$. Consecutively differentiating the left-hand side of~\eqref{iso-sphere} two times with respect to $t$, substituting $\dot{x}(0)=v$, $\dot{y}(0)=-u$, $\ddot{x}(0)=u$, $\ddot{y}(0)=v$, $R=0$, and solving the resulting system of $3$ linear equations in $m_1,m_2,m_3$, we get~\eqref{eq-prop-vertex}.
\end{proof}

The side which the cone $C$ borders the tangent plane $P$ upon at the point $\mathbf{r}$ can be identified as follows.

\begin{proposition}
Let a plane $P$ and a cone $C$ be such that $P^i=(x,y,z)$ and the conic $C^i$ is parametrized by~\eqref{eq-conj-reciprocal}. In particular, the vector $\mathbf{n}(x,y)$ given by~\eqref{eq-inverse-stereographic} is normal to $P$. Let $\mathbf{m}$ be the vertex of $C$ and $\mathbf{r}$ be a tangency point of $C$ and $P$. Take any point $(x',y')$ in the top view of $C^i$ distinct from $(x,y)$; e.g., $(x',y') =(x+2u, y+2v)$. Then $C$ borders upon $P$ at $\mathbf{r}$ from the side of the halfspace containing
$$
\begin{cases}
\mathbf{n}(x,y),
&\text{if }\mathbf{n}(x',y')\cdot (\mathbf{r}-\mathbf{m}) > 0;\\
-\mathbf{n}(x,y),
&\text{if }\mathbf{n}(x',y')\cdot (\mathbf{r}-\mathbf{m}) < 0.
\end{cases}
$$
\end{proposition}

\begin{proof}
For	$(x',y') =(x+2u, y+2v)$ cut the cone $C$ by the plane passing through $\mathbf{m}$ and $\mathbf{r}$ and being parallel to $\mathbf{n}(x,y)$. Then the plane is parallel to the vector $\mathbf{n}(x',y')$ as well, and the proposition reduces to an obvious planar problem. For other	$(x',y')\ne (x,y)$ the proposition follows by the continuity.
\end{proof}


\begin{remark}
  The distance between the vertex of the cone mapped to conic~\eqref{eq-conj-reciprocal} and the tangency point with the surface $\Phi$ mapped to $z=f(x,y)$ equals $|\mathbf{m}(x,y)-\mathbf{r}(x,y)|$,
  where $\mathbf{r}(x,y)$ and $\mathbf{m}(x,y)$ are given by \eqref{eq-cor1} and \eqref{eq-prop-vertex}.  Adding the inequality
  $$
  r\le |\mathbf{m}(x,y)-\mathbf{r}(x,y)|\le R
  $$
  to equations~\eqref{eq-conj-characterization1},\eqref{eq-conj-characterization2},\eqref{eq-conj-characterization3} we get a necessary condition of the surface $\Phi$ to be the envelope of one-parameter family of cones with the opening angle $\theta$, \emph{truncated} at distances $r$ and $R$ from the vertex.
\end{remark}

\section{Envelopes of congruent rotational cylinders} \label{sec:cylinder}

Cylinders are a limit case of cones, but this limit is not straightforward. This is so, since the limit of
cones with a constant opening angle are cones with vanishing opening angle, i.e., rotational cylinders. However,
these cylinders need not be congruent. Hence, we now discuss envelopes of congruent rotational cylinders, i.e. offsets of ruled surfaces,
which appear in flank CNC machining with a cylindrical tool.

The derivation of the PDE is analogous to Sections~\ref{sec-model}--\ref{sec-conics}. Passing to the isotropic model, we reduce the characterization of surfaces in question to the characterization of surfaces containing a special conic through each point.
We parametrize the conic by trigonometric functions and identify the particular conditions on the conic. Differentiation with respect to the parameter gives the required PDE.

\begin{proposition}\label{prop-offset-ruled}
Assume (*). Through each point of $\Phi$ there passes an oriented cylinder of fixed radius $R$ which is tangent to $\Phi$ along a continuous curve (containing the point), has inwards oriented normals, and the axis nonparallel to the plane $z=0$, if and only if through each point of the surface $\Phi^i$ there passes an arc of a conic satisfying the following condition:

$(R)$ the top view of the conic is the stereographic projection of a great circle (not passing through the projection center $(0,0,-1)$), and the plane of the conic passes through the point $(0,0,R)$.
\end{proposition}

\begin{remark} \label{rem-outwards} A similar propositions holds for a cylinder with outwards oriented normals, only $(0,0,R)$ is replaced by $(0,0,-R)$.
\end{remark}

\begin{proof}
Let $C$ be an oriented cylinder of radius $R$ with inwards oriented normals and the axis nonparallel to the plane $z=0$. The oriented tangent planes to $C$ are the common oriented tangent planes of some two oriented spheres $S_1$ and $S_2$ of radius $R$ with inwards oriented normals. Then $C^i$ is the intersection of $S_1^i$ and $S_2^i$. Assume that $S_1$ is contained in the halfspace $z\le 0$ and tangent to the plane $z=0$. Then $S_2^i$ is a paraboloid of form~\eqref{iso-sphere}, and $S_1^i$ is a plane. Hence $C^i$ is a conic.
Since the oriented sphere $S_1$ is tangent to the oriented plane $P$ given by $z=-2R$ with the normal $(0,0,1)$, by~\eqref{eq-isotropic-model} it follows that the plane $S_1^i$ of the conic passes through the point
$P^i=(0,0,R)$. The top view of $C^i$ is the stereographic projection of the Gaussian spherical image of $C$, i.e., the projection of a great circle. Now if $C$ is tangent to $\Phi$ along a curve (which cannot be a ruling because by (*) $\Phi$ has nonvanishing Gaussian curvature), then $C^i$ is contained in $\Phi^i$. The proof of the reciprocal implication is analogous.
\end{proof}

\begin{proposition}\label{prop-conic-plane}
  Consider conic~\eqref{eq-conj-reciprocal}, where $a,b,z$
  are given by~\eqref{eq-prop-tangency-order2} for some $C^2$ function $f\colon D\to \mathbb{R}$. Then the conic satisfies condition~$(R)$, if and only if the following two equations hold:
  \begin{gather}
  x^2+y^2+1+2xu+2yv=0,\label{eq-conj-characterization1zero}
  \\
  2(u^2+v^2)(f-xf_x-yf_y-R)+(x^2+y^2+1)(f_{xx} v^2 - 2f_{xy}uv + f_{yy}u^2)=0.
  \label{eq-prop-conic-plane}
  \end{gather}
\end{proposition}

\begin{proof} Since a great circle in the unit sphere has intrinsic radius $\pi/2$,
substituting $\theta=0$ into~\eqref{eq-conj-characterization1}, we get~\eqref{eq-conj-characterization1zero}. If the plane of the conic passes through $(0,0,R)$, we get $z(t)=Ax(t)+By(t)+R$ for some constants $A,B\in\mathbb{R}$. Hence
$$
\begin{cases}
z=Ax+By+R, & \\
a=Av-Bu,   & \\
b=Au+Bv.   &
\end{cases}
$$
The latter two equations in $A$ and $B$ are linearly independent because $(u,v)\ne (0,0)$ by~\eqref{eq-conj-characterization1zero}.
Thus the system has a solution $(A,B)$, if and only if
$\det\left(\begin{smallmatrix}
x & y  & z-R\\
v & -u & a  \\
u & v  & b
\end{smallmatrix}\right)=0$.
Using~\eqref{eq-prop-tangency-order2} and~\eqref{eq-conj-characterization1zero}, we get~\eqref{eq-prop-conic-plane}.
\end{proof}

Combining Propositions~\ref{prop-offset-ruled}, \ref{prop-tangency-order}, and~\ref{prop-conic-plane} we get the following result.

\begin{corollary}[recognition of ruled surface offsets] \label{cor-offset-ruled}
Assume (*). If through each point of $\Phi$ there passes an oriented cylinder of fixed radius $R$ which is tangent to $\Phi$ along a continuous curve (containing the point), has inwards oriented normals and the axis nonparallel to the plane $z=0$,
then for each $(x,y)\in D$ the three equations~\eqref{eq-conj-characterization2}, \eqref{eq-conj-characterization1zero},
\eqref{eq-prop-conic-plane} have a common real solution $(u,v)$.
\end{corollary}

We keep just $3$ equations in $2$ variables $u$ and $v$ because it is already a nontrivial restriction on the function~$f$.

\section{Channel surfaces and pipe surfaces} \label{sec:sphere}

For the sake of completeness and due to the similarity with our previous results, we also address
envelopes of spheres. We first consider the case in which the spheres may have varying radius,
leading to channel surfaces as envelopes. We then specialize to constant radius and pipe
surfaces.

\begin{proposition}\label{prop-channel}
Assume~(*). Then the following three conditions are equivalent:
\begin{itemize}
  \item through each point of $\Phi$ there passes an oriented sphere which is tangent to $\Phi$ along a circular arc (containing the point) but is not tangent to the oriented plane $z=0$ with the normal $(0,0,-1)$, and all the oriented unit normals to the sphere on the circle are distinct from $(0,0,-1)$;
  \item through each point of $\Phi^i$ there passes a rotational paraboloid with vertical axis which is tangent to $\Phi^i$ along an arc of a conic (containing the point) distinct from a parabola;
  \item through each point of $\Phi^i$ there passes an arc of a conic contained in $\Phi^i$ and satisfying the following condition:

      $(T)$ the top view of the conic is a circle, the tangent planes to $\Phi^i$ along the arc of the conic have a unique common point, and the top view of the resulting point is the center of the circle.
\end{itemize}
\end{proposition}

\begin{proof}
The first condition implies the second one as follows. Indeed,  by~\eqref{iso-sphere} the oriented sphere is represented by a rotational paraboloid with a vertical axis in the isotropic model. The tangent planes along a circular arc on a sphere form a cone or a cylinder, hence are represented by a conic (actually an isotropic circle) on the paraboloid, by Propositions~\ref{prop-cone-translation}, \ref{prop-offset-ruled}, and Remark~\ref{rem-outwards}. Since the map $\Phi\mapsto \Phi^i$ preserves oriented tangency, the paraboloid is tangent to $\Phi^i$ along an arc of the conic.

The second condition implies the third one because each conic (distinct from a parabola) on a rotational paraboloid with vertical axis satisfies condition~(T) (with $\Phi^i$ replaced by the paraboloid). Indeed, any nonvertical planar section of such paraboloid can be made horizontal by an affine map of the form $(x,y,z)\mapsto (x,y,z-ax-by)$, and for a horizontal section condition~(T) obviously holds by rotational symmetry.

The proofs of reciprocal implications are analogous (one cannot end up with a sphere of radius zero because $\Phi$ is smooth by (*)).
\end{proof}

\begin{proposition}\label{prop-conic-cone}
  Consider conic~\eqref{eq-conj-reciprocal}, where $a,b,z$
  are given by~\eqref{eq-prop-tangency-order2} for some $C^2$ function $f\colon D\to \mathbb{R}$. If  an arc of the conic parametrized by $t\in(-\epsilon,\epsilon)$ is contained in the graph $\Phi^i$ of $f$ and satisfies condition~~$(T)$, then
  \begin{equation}\label{eq-prop-conic-cone}
  (f_{xx}-f_{yy}) uv + f_{xy}(v^2-u^2)=0.
  \end{equation}
\end{proposition}


\begin{proof}
  Let $(x+u,y+v,z_0)$ be the common point of the tangent planes to the graph of $f\colon D\to \mathbb{R}$ along the arc of the conic~\eqref{eq-conj-reciprocal}. Then  the line through the variable point $(x(t),y(t),z(t))$ on the arc and the common point lies on the tangent plane at the variable point. Hence
  $$
  z(t)-z_0=f_x\cdot(x(t)-x-u)+f_y\cdot(y(t)-y-v).
  $$
  Differentiating with respect to $t$ and substituting $t=0$,
  ${x}(0)=x$, $y(0)=y$, $\dot{x}(0)=v$, $\dot{y}(0)=-u$, $\dot{z}(0)=a=f_x v-f_y u$, given by Proposition~\ref{prop-tangency-order}, we arrive at~\eqref{eq-prop-conic-cone}.
\end{proof}

\begin{corollary}[recognition of channel surfaces] \label{cor-channel}
  Assume (*). If through each point of $\Phi$ there passes an oriented sphere which is tangent to $\Phi$ along a circular arc (containing the point) but is not tangent to the oriented plane $z=0$ with the normal $(0,0,-1)$, and the oriented unit normals to the sphere on the circle are distinct from $(0,0,-1)$, then for each $(x,y)\in D$ the following two equations have a common nonzero real solution~$(u,v)$:
  \begin{equation}\label{eq-cor-channel}
  \begin{cases}
    (f_{xx}-f_{yy}) uv + f_{xy}(v^2-u^2)=0, &  \\
    f_{xxx}v^3-3f_{xxy}v^2u+3f_{xyy}vu^2-f_{yyy}u^3=0. &
  \end{cases}
  \end{equation}
\end{corollary}

\begin{proof}
  By Propositions~\ref{prop-channel} and \ref{prop-conic-cone} the first equation in~\eqref{eq-cor-channel} follows.
  By Propositions~\ref{prop-channel} and~\ref{prop-tangency-order} we have~\eqref{eq-conj-characterization2}. Subtracting \eqref{eq-prop-conic-cone} multiplied by $3$, we get the second equation in~\eqref{eq-cor-channel}.
\end{proof}

Since each pipe surface is also an offset of a ruled surface, by Propositions~\ref{prop-offset-ruled}--\ref{prop-conic-cone} we get the following corollary.

\begin{corollary}[recognition of pipe surfaces] \label{cor-pipe}
  Assume (*). If through each point of $\Phi$ there passes an oriented sphere of fixed radius $R$ with inwards oriented normals, which is tangent to $\Phi$ along a circular arc (containing the point) of the same radius $R$ but is not tangent to the oriented plane $z=0$ with the normal $(0,0,-1)$, and the plane of the circle is nonparallel to $(0,0,-1)$, then for each $(x,y)\in D$ two equations~\eqref{eq-prop-conic-plane} and~\eqref{eq-prop-conic-cone} have a common nonzero real solution $(u,v)$.
\end{corollary}

\begin{problem} \label{prob-widely-open} Are the reciprocal assertions in Corollaries~\ref{cor-offset-ruled}, \ref{cor-channel}, \ref{cor-pipe} true (if ``each point'' is replaced by ``a generic point'')?
\end{problem}

\section{Results and applications in CNC machining}\label{sec-results}

In this section, we show how the proposed analysis of third order contact can be used in the context of 5-axis flank CNC machining with
conical tools. First, we test our algorithm on an exact envelope, showing that we reconstruct the generators of the envelope.

\begin{example} \emph{Reconstruction of an exact generator}.
In the isotropic space, consider the graph of the  function $f(x,y)=\frac{y^2}{x^2+y^2}$. The graph contains a family of isotropic circles whose top views are Euclidean circles passing through the origin $(0,0)$ and having radius  $\frac{1}{2}\cot\theta$. They are the stereographic projections of circles of intrinsic radius $\frac{\pi}{2} - \theta$ passing through $(0,0,1)$ on the unit sphere.
In the design space, they correspond to a motion of a cone with the opening angle $\theta=30^{\circ}$, see Fig.~\ref{fig:IsotropicF}. We validated our approach by reconstructing the exact generator.
Observe that there are two positions of the generating cone (as there are two isotropic circles passing through the point of the graph). One of the generating cones is shown in yellow in Fig.~\ref{fig:IsotropicF}.

 \begin{figure}[!tb]
\vrule width0pt \hfill
\begin{overpic}[width=0.32\columnwidth]{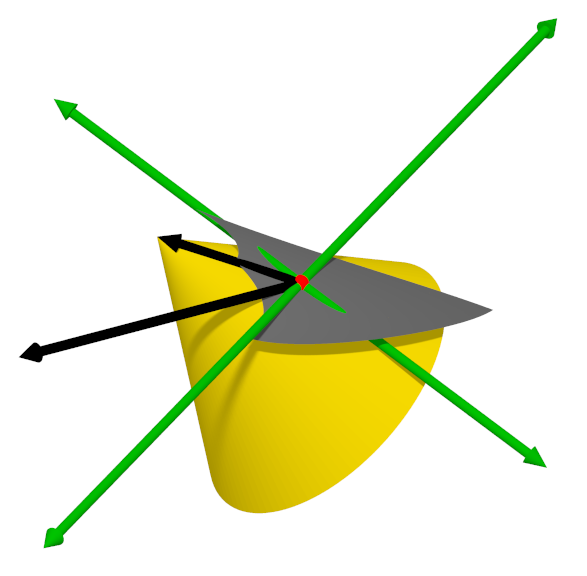}
{\small
	\put(0,-2){(a)}
	\put(23,57){$v$}
	\put(80,48){$\Phi$}
  }
\end{overpic}
\hfill
\begin{overpic}[width=0.32\columnwidth]{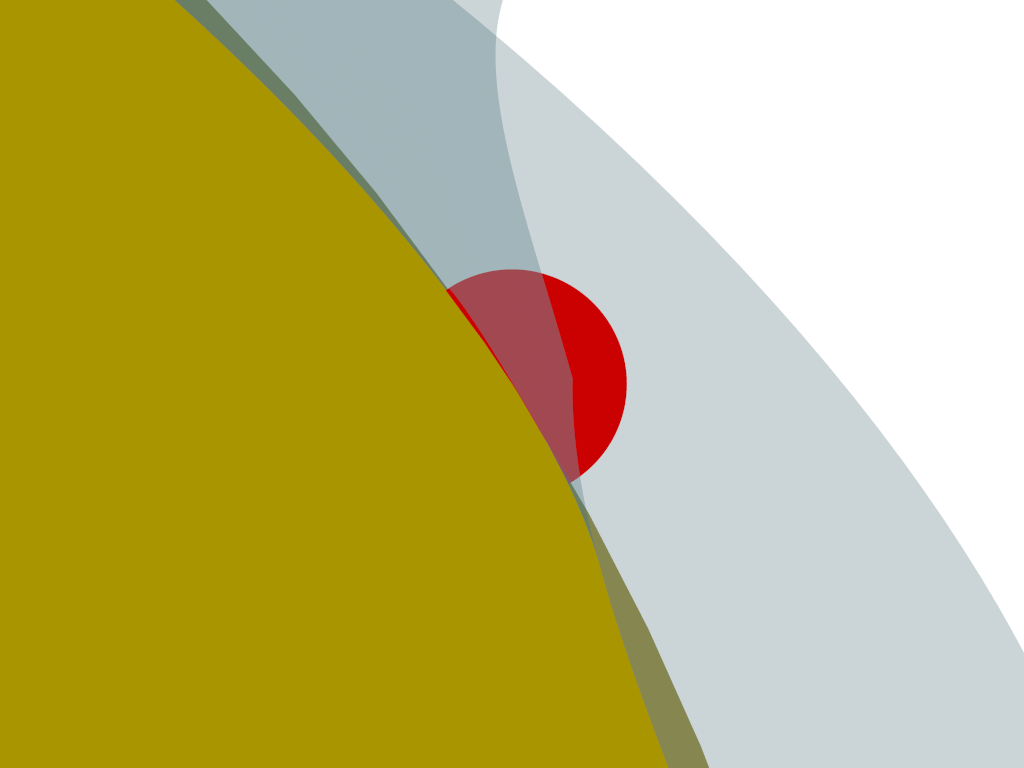}
{\small
  \put(-8,-2){(b)}
  }
\end{overpic}
\hfill\vrule width0pt\\
\vspace{-12pt}
\caption{Reconstruction of an exact envelope. (a) An exact envelope $\Phi$ (dark) is generated from the function $f(x,y)=\frac{y^2}{x^2+y^2}$ by applying \eqref{eq-cor1}. The candidate tangent directions, i.e., the rulings, where third order contact with osculating cones occurs, are computed via~\eqref{eq-simplified-system},  \eqref{eq-cor1}, and~\eqref{eq-prop-vertex}. The endpoints of the vectors correspond to the vertices of the hyperosculating cones. Our algorithm detects six positions with third order contact (green), including two rulings (black) that correspond to the exact generators. One generating cone (yellow) with the opening angle $\theta=30^{\circ}$ is shown.
(b) A zoom-in view from the vertex $v$ of the cone. Observe the perfect local match between the cone and the surface (rendered in transparent) in a neighborhood of the contact point (red).}
\label{fig:IsotropicF}
\end{figure}

\end{example}

To further validate our approach, we tested to what extent one may have inexact data, and yet reconstruct the exact solution.

\begin{example}
\emph{Stability}. We conducted a stability test as follows. We took the exact envelope generated from the function $f(x,y)=\frac{y^2}{x^2+y^2}$ by applying~\eqref{eq-cor1}, perturbed the sampled tangent planes, and mapped these planes back to the isotropic space, see Fig.~\ref{fig:Stability}. The tangent plane perturbation was achieved by adding a random noise to the surface normals as follows. Let $\lbrace\bd_1,\bd_2,\bn\rbrace$ be an orthonormal frame at a contact point, $\bn$ being the unit normal. We define $\bv = \alpha_1 \bd_1 + \alpha_2 \bd_2$ and the perturbed normal as $\bnt=\frac{\bn+\bv}{\Vert \bn+\bv \Vert}$, where $\alpha_1 = r \cos(\phi)$, $\alpha_2 = r \sin(\phi)$. The angle $\phi$ is randomly sampled from $[-\pi, \pi]$ and the random deviation is controlled via the parameter $r$ which is set to $r=0.1$ in the example shown in Fig.~\ref{fig:Stability}(a). The reconstruction of the isotropic circles from the exact and perturbed data are shown Fig.~\ref{fig:Stability}(b) and reconstruction of the hyperosculating cones is shown in Fig.~\ref{fig:Stability}(c).

 \begin{figure}[!tb]
\vrule width0pt \hfill
\begin{overpic}[width=0.39\columnwidth]{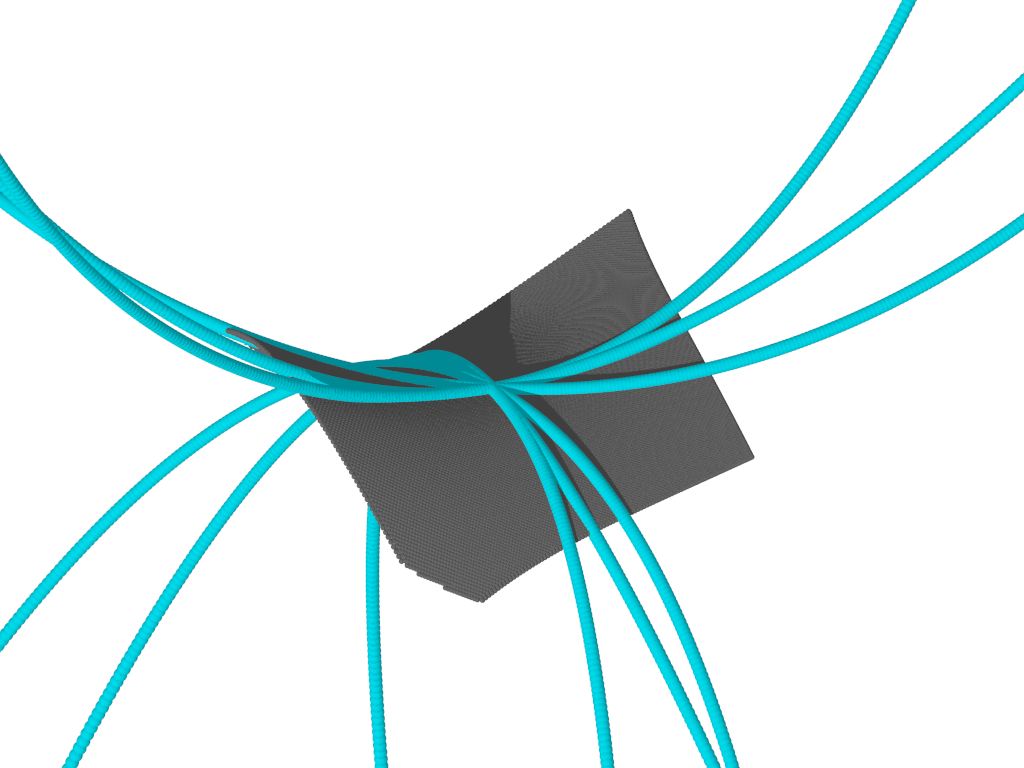}
{\small
\put(20,50){\fcolorbox{gray}{white}{\includegraphics[width=0.15\textwidth]{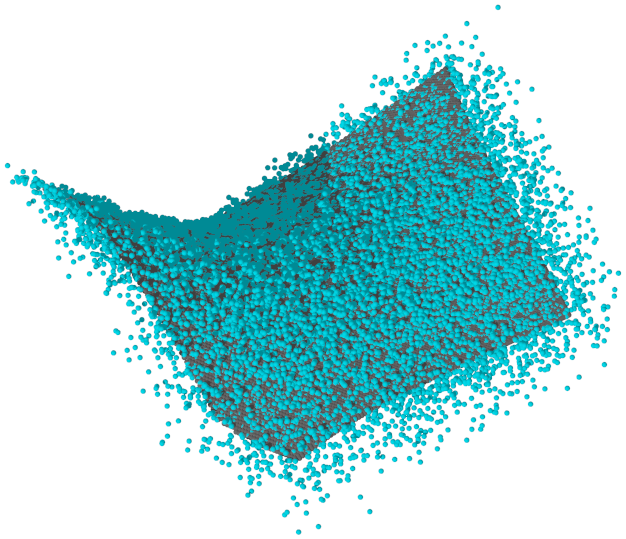}}}
\put(73,25){$\Phi^i$}
  \put(0,-2){(a)}
  }
\end{overpic}
\hfill
\begin{overpic}[width=0.17\columnwidth]{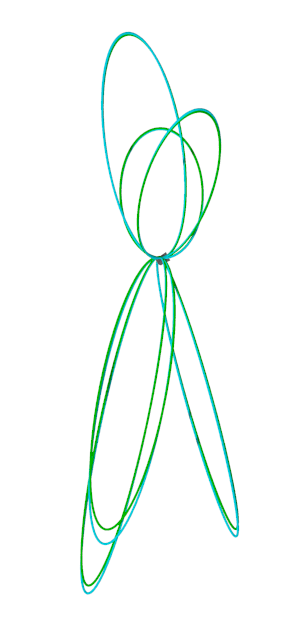}
{\small
	\put(0,-2){(b)}
	\put(30,55){$\Phi^i$}
  }
\end{overpic}
\hfill
\begin{overpic}[width=0.35\columnwidth]{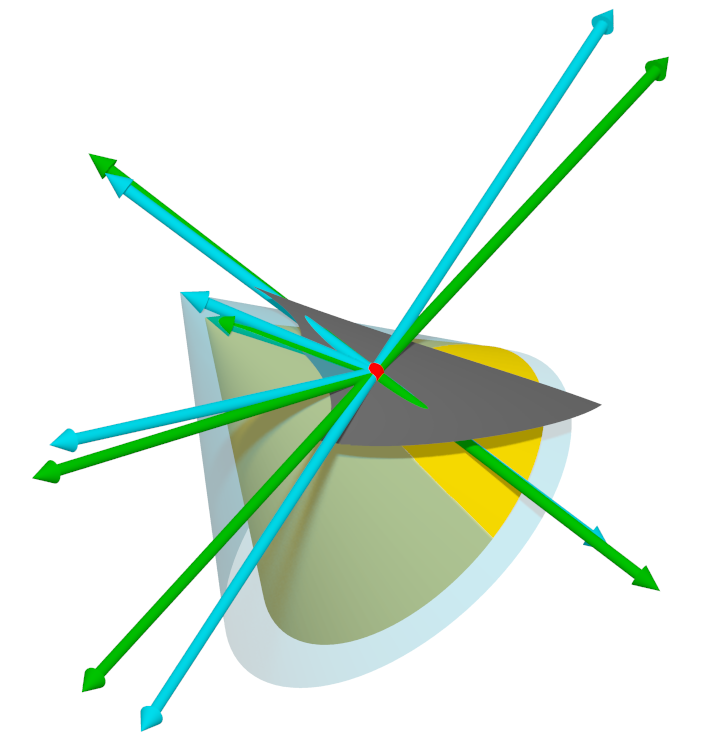}
{\small
  \put(0,-2){(c)}
  }
\end{overpic}
\hfill\vrule width0pt\\
\vspace{-12pt}
\caption{Stability. (a) The tangent planes
of the surface $\Phi$ were generated by the function  $f(x,y)=\frac{y^2}{x^2+y^2}$,  perturbed, and mapped back to the isotropic space, returning a noisy point cloud (top framed). The six hyperosculating isotropic circles arising from the noisy data are shown. (b) A zoom-out of the isotropic circles arising from the exact data (green) and from the noisy one (blue). (c) The situation in the design space: the hyperosculating tangent directions from exact (green) and perturbed (blue) data. Two hyperosculating cones that correspond to the surface generator (yellow) and its approximation from noisy data (transparent) are shown.}
\label{fig:Stability}
\end{figure}

\end{example}

\begin{example}
\emph{Industrial benchmark}. The hyperosculating configurations, see Fig.~\ref{fig:Impeller}, can be used for initialization of path-planning algorithm of 5-axis flank CNC machining with conical milling tools \cite{Calleja-2018-FlankMillConicalTools}. Observe that only some hyperosculating cones can be used as candidates for the tool position due penetration of the cone with the neighboring blades. A sequence of positions at several hyperbolic contact points is shown in Fig.~\ref{fig:Impeller}(h). However, the distance between the vertex of the cone and the contact point, i.e., the tool size, varies. These issues as well as selection of a suitable opening angle go beyond this paper.

 \begin{figure}[!tb]
\vrule width0pt \hfill
\begin{overpic}[width=0.31\columnwidth]{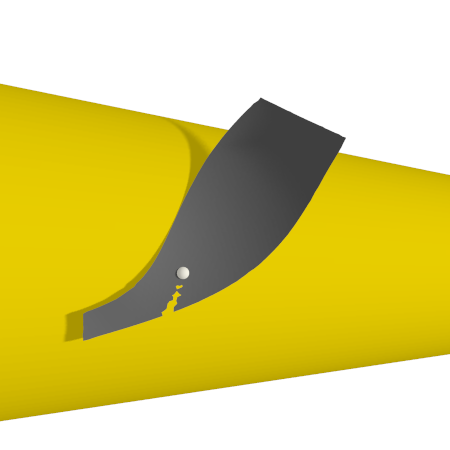}
{\small
\put(5,55){\fcolorbox{gray}{white}{\includegraphics[width=0.12\textwidth]{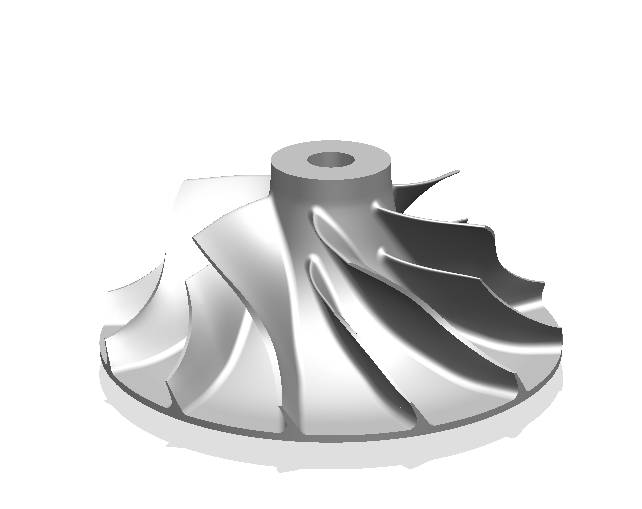}}}
  \put(0,2){(a)}
  }
\end{overpic}
\hfill
\begin{overpic}[width=0.31\columnwidth]{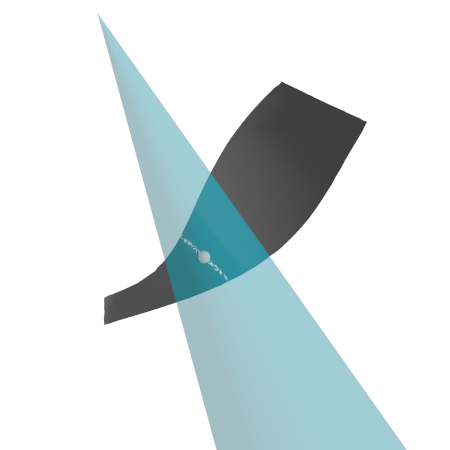}
{\small
	\put(0,2){(b)}
  }
\end{overpic}
\hfill
\begin{overpic}[width=0.31\columnwidth]{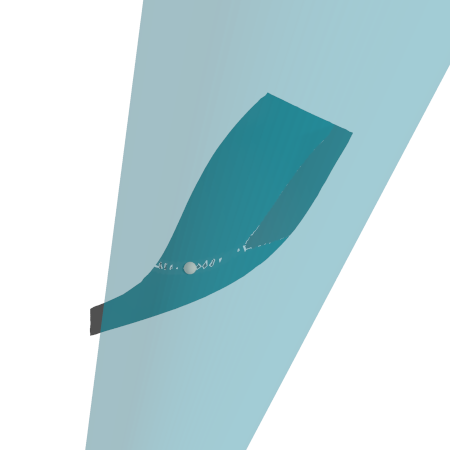}
{\small
  \put(0,2){(c)}
  }
\end{overpic}
\hfill\vrule width0pt\\
\vrule width0pt \hfill
\begin{overpic}[width=0.31\columnwidth]{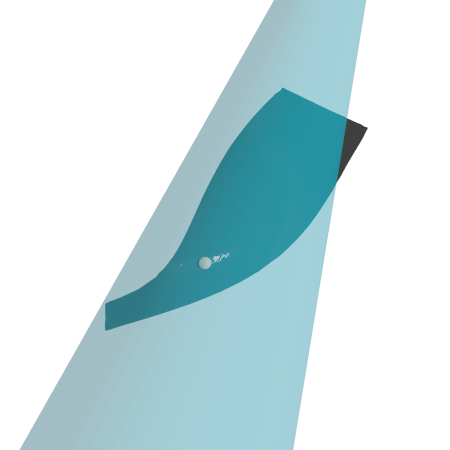}
{\small
  \put(0,2){(d)}
  }
\end{overpic}
\hfill
\begin{overpic}[width=0.31\columnwidth]{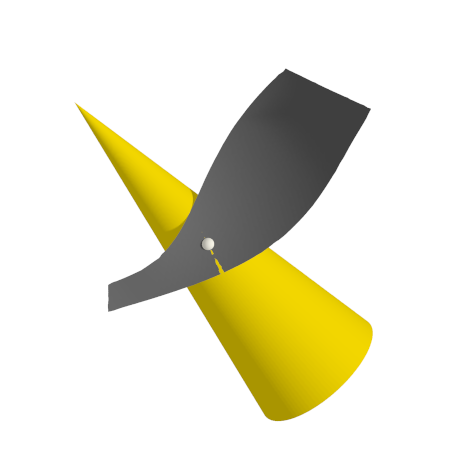}
{\small
	\put(0,2){(e)}
  }
\end{overpic}
\hfill
\begin{overpic}[width=0.31\columnwidth]{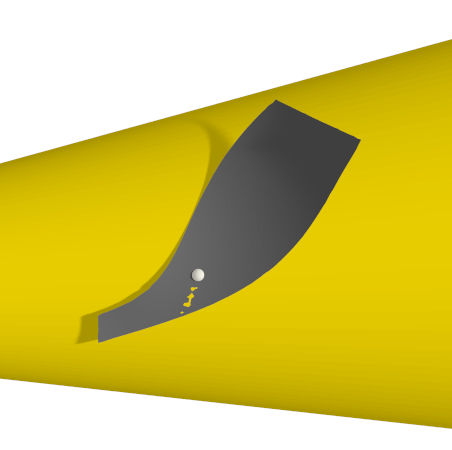}
{\small
  \put(0,2){(f)}
  }
\end{overpic}
\hfill\vrule width0pt\\
\vrule width0pt \hfill
\begin{overpic}[width=0.53\columnwidth]{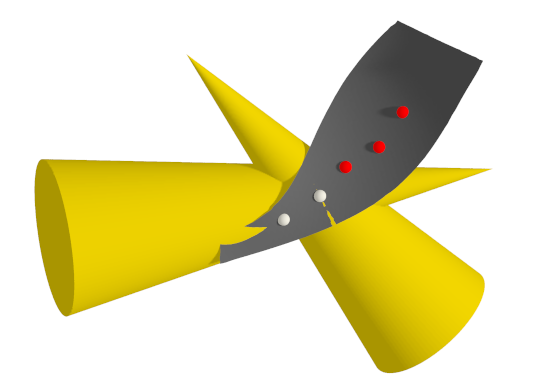}
{\small
  \put(0,2){(g)}
  }
\end{overpic}
\hfill
\begin{overpic}[width=0.42\columnwidth]{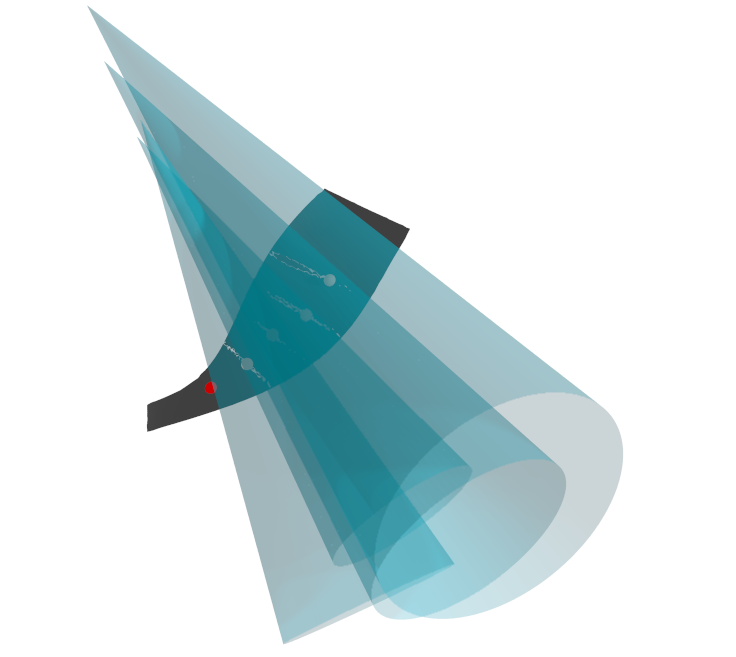}
{\small
	\put(0,2){(h)}
  }
\end{overpic}
\hfill\vrule width0pt\\
\vspace{-12pt}
\caption{Industrial data. (a-f) Six positions of a cone with a specific opening angle $\theta=10^{\circ}$ that hyperosculates a blade of the impeller (framed) at contact point (white) are shown. The yellow cones approximate the blade ``from-below'', the transparent cones ``from-above''. Due to collision with other blades, only configurations (b) and (e) can be considered for flank CNC machining. (g) At five, user-selected, contact points, the hyperosculating cones are computed and suitable cones are shown. The red points indicate that there are no suitable hyperosculating cones from-below. (h) The sequence of hyperosculating cones from-above.}
\label{fig:Impeller}
\end{figure}

\end{example}

\section{Conclusion and future research}\label{sec-conclu}

We have derived necessary and sufficient conditions on a surface to be an  envelope of a one-parameter family of
congruent rotational cones. Such a surface can be milled by flank CNC machining with an appropriate conical tool
in a single trace (provided that the motion is collision free and technical constraints on available tool sizes,
machine workspace etc. are fulfilled as well).
This characterization comes in form of nonlinear PDEs.
On our way towards that, we discussed similar PDEs for developables surfaces and ruled surfaces, and for completeness,
we addressed envelopes of cylinders and spheres as well.

The derivation of our results uses contact order in the space of planes and the isotropic model of Laguerre
geometry. It also led to conditions for higher order contact between rotational cones and a general surface,
but contact is meant here in the space of planes. Contact in the standard point sense would not be useful anyway.
We have shown (Section \ref{ssec-contact-order}) that contact order in the space of planes is the right
concept to get higher order contact between a surface generated by a conical (or cylindrical) tool and the
target surface to be machined.

Finally, we provided some initial computational results which indicate that the concepts are useful for the
development of algorithms for path planning in 5-axis flank CNC machining with conical tools. This is the
part where future research could continue. The high order contact positions found according to our results
should serve as appropriate initial tool positions for further optimization. A main research goal to be addressed
is a complete coverage of a given design surface by well fitting envelopes of a moving tool, keeping the
so-called scallop heights between neighboring machined strips as small as possible. Ideally, one could try
to obtain scallop height free results in the sense that neighboring machined strips join smoothly.
 Even direct surface design could be guided by the fabrication with a certain technology, especially when
 very high accuracy is required. This would amount to the design of surfaces composed of surface strips which
 can be produced precisely with a certain technology.

\section*{Acknowledgements}
Three of the four coauthors are grateful to King Abdullah University of Science and Technology, where they met altogether and started this project.
The authors are also grateful to R.~Bryant, S.~Ivanov, and A.~Skopenkov for useful discussions.
The first author has been supported within the framework of the Academic Fund Program at the National Research University Higher School of Economics (HSE) in 2018-2019 (grant N18-01-0023) and by the Russian Academic Excellence Project ``5-100''.
The second author has been partially supported by the National Natural Science Foundation of China (61672187) and the Shandong Provincial Key R$\&$D Program (2018GGX103038).
The third author has been partially supported by
Spanish Ministry of Science, Innovation and Universities: Ram\'{o}n y Cajal with reference RYC-2017-22649 and the European Union’s Horizon 2020 research and innovation programme under agreement No. 862025.




\bibliographystyle{amsplain}
\bibliography{envelopes,Machining}

\end{document}